\documentclass[reqno]{amsart}

\usepackage[latin1]{inputenc}
\usepackage{hyperref}
\usepackage{amssymb}

\newcommand{\interieur}[1]{\ensuremath{\operatorname{int} #1}}
\newcommand{\adherence}[1]{\ensuremath{\overline{#1}}}	

\newcommand{\intn}[2]{\ensuremath{\{  #1  ,\ldots,  #2 \}}}

\def\grint{\mathrm{G}}
\def\grintfalc{\mathcal{G}}

\def\R{\mathbb{R}}
\def\Q{\mathbb{Q}}
\def\N{\mathbb{N}}
\def\Z{\mathbb{Z}}

\def\hau{\mathcal{H}}
\def\leb{\mathcal{L}}

\def\S{\mathcal{S}}
\def\I{\mathcal{I}}

\def\netm{\mathcal{M}}
\newcommand{\diam}[1]{\ensuremath{| #1 |}}
\newcommand{\gene}[1]{\ensuremath{\langle #1 \rangle}}

\def\eps{\varepsilon}
\def\ph{\varphi}

\def\jauge{\mathfrak{D}}
\def\prem{\mathbb{P}}

\def\weak{\rm w}

\newcommand{\modu}[1]{\ensuremath{\,\operatorname{mod}\, #1}}
\def\devel{{\rm Dev}}
\def\birkhoff{{\rm Bir}}
\def\mesb{\mathfrak{M}}
\newcommand{\restr}[1]{_{\rule[-.41ex]{.03em}{2ex}#1}}
\def\ind{\mathrm{1}}

\renewcommand{\theenumi}{\alph{enumi}}

\theoremstyle{plain}
\newtheorem{thm}{Theorem}
\newtheorem{prp}{Proposition}
\newtheorem{lem}[prp]{Lemma}
\newtheorem{cor}[prp]{Corollary}

\theoremstyle{definition}
\newtheorem*{df}{Definition}

\theoremstyle{remark}
\newtheorem*{rem}{Remark}

\title{Ubiquitous systems and metric number theory}
\author{Arnaud Durand}
\subjclass[2000]{Primary 28A80; Secondary 28A78, 11J83}
\address{Laboratoire d'Analyse et de Mathématiques Appliquées, Université Paris XII, 61 av. du Général de Gaulle, 94010 Créteil Cedex, France.}
\email{a.durand@univ-paris12.fr}

\begin{document}

\begin{abstract}
We investigate the size and large intersection properties of
$$
E_{t}=\{ x\in\R^d \:|\: \|x-k-x_{i}\|<{r_{i}}^t\text{ for infinitely many }(i,k)\in I^{\mu,\alpha}\times\Z^d\},
$$
where $d\in\N$, $t\geq 1$, $I$ is a denumerable set, $(x_{i},r_{i})_{i\in I}$ is a family in $[0,1]^d\times (0,\infty)$ and $I^{\mu,\alpha}$ denotes the set of all $i\in I$ such that the $\mu$-mass of the ball with center $x_{i}$ and radius $r_{i}$ behaves as ${r_{i}}^\alpha$ for a given Borel measure $\mu$ and a given $\alpha>0$. We establish that the set $E_{t}$ belongs to the class $\grint^h(\R^d)$ of sets with large intersection with respect to a certain gauge function $h$, provided that $(x_{i},r_{i})_{i\in I}$ is a heterogeneous ubiquitous system with respect to $\mu$. In particular, $E_{t}$ has infinite Hausdorff $g$-measure for every gauge function $g$ that increases faster than $h$ in a neighborhood of zero. We also give several applications to metric number theory.
\end{abstract}

\maketitle

\section{Introduction}\label{grintheterintro}

A classical problem in the theory of Diophantine approximation is to describe the size properties of the set
$$
J_{\tau}=\left\{ x\in\R \:\Biggl|\: \left|x-\frac{p}{q}\right|<q^{-\tau}\text{ for infinitely many }(p,q)\in\Z\times\N \right\}
$$
of all real numbers that are $\tau$-approximable by rationals ($\tau>0$). A well-known theorem of Dirichlet ensures that $J_{\tau}=\R$ if $\tau\leq 2$, see~\cite{Hardy:1979fk}. Conversely, $J_{\tau}$ has Lebesgue measure zero if $\tau>2$. To give a more precise account of the size properties of $J_{\tau}$, Jarn\'ik~\cite{Jarnik:1929mf} and Besicovitch~\cite{Besicovitch:1934ly} established that its Hausdorff dimension is $2/\tau$. Furthermore, Jarn\'ik~\cite{Jarnik:1931qf} determined its Hausdorff $h$-measure (see Section~\ref{grinthetersets} for the definition) for some functions $h$ in the set $\jauge_{1}$ which is defined as follows.

\begin{df}
For $d\in\N$, let $\jauge_{d}$ be the set of all functions which vanish at zero, are continuous and nondecreasing on $[0,\eps]$ and are such that $r\mapsto h(r)/r^d$ is positive and nonincreasing on $(0,\eps]$, for some $\eps>0$. Any function in $\jauge_{d}$ is called a gauge function. Moreover, for $g,h\in\jauge_{d}$, let us write $g\prec h$ if $g/h$ monotonically tends to infinity at zero.
\end{df}

The result of Jarn\'ik, later refined by V.~Beresnevich, D.~Dickinson and S.~Velani~\cite{Beresnevich:2006ve}, is the following: for any $h\in\jauge_{1}$ with $h(r)\prec r$, the set $J_{\tau}$ has infinite (resp. zero) Hausdorff $h$-measure if $\sum_{q} h(q^{-\tau})q=\infty$ (resp. $<\infty$). On top of that, K.~Falconer~\cite{Falconer:1994hx} proved that $J_{\tau}$ enjoys a {\em large intersection} property, in the sense that it belongs to the class $\grintfalc^{2/\tau}(\R)$. Recall that the class $\grintfalc^s(\R^d)$ of sets with large intersection of dimension at least a given $s\in (0,d]$ was defined by K.~Falconer~\cite{Falconer:1994hx} as the collection of all $G_{\delta}$-subsets $F$ of $\R^d$ such that
$$
\dim\bigcap_{n=1}^\infty f_{n}(F)\geq s
$$
for every sequence $(f_{n})_{n\in\N}$ of similarities, where $\dim$ stands for Hausdorff dimension. It is the maximal class of $G_{\delta}$-sets of dimension at least $s$ that is closed under countable intersections and similarities. Thus, any intersection of a countable number of sets of $\grintfalc^s(\R^d)$ is of dimension at least $s$. Note that this property is quite counterintuitive, since the dimension of the intersection of two subsets of $\R^d$ of dimensions $d_{1}$ and $d_{2}$ respectively is usually expected to be $d_{1}+d_{2}-d$, see~\cite[Chapter~8]{Falconer:2003oj}.

Theorem~5 in~\cite{Durand:2007uq} describes more precisely the large intersection properties of $J_{\tau}$. Specifically, $J_{\tau}$ belongs to a certain class $\grint^h(\R)$ for any $h\in\jauge_{1}$ with $\sum_{q} h(q^{-\tau})q=\infty$. The class $\grint^h(\R)$, defined in~\cite{Durand:2007uq}, is closed under countable intersections and similarities and every set in it has infinite Hausdorff $g$-measure for any $g\in\jauge_{1}$ such that $g\prec h$ (see Section~\ref{grinthetersets} for details and complements). In particular, $J_{\tau}$ belongs to $\grint^{r^{2/\tau}}(\R)$, which is included in the class $\grintfalc^{2/\tau}(\R)$ of K.~Falconer. Moreover, Theorem~5 in~\cite{Durand:2007uq} ensures that this set has maximal Hausdorff $h$-measure in every open subset of $\R$ for any $h\in\jauge_{1}$ such that $\sum_{q} h(q^{-\tau})q=\infty$. This is an improvement on Jarn\'ik's theorem. The aforementioned results follow from a certain homogeneity in the repartition of the rationals: they form a {\em homogeneous ubiquitous system} in the sense of~\cite{Durand:2007uq}.

Further improvements are obtained by allowing restrictions on the rational approximates. For instance, G.~Harman~\cite{Harman:1998fk} studied the set
$$
J^\prem_{\tau}=\left\{ x\in\R \:\Biggl|\: \left|x-\frac{p}{q}\right|<q^{-\tau}\text{ for infinitely many primes }(p,q)\in\Z\times\N \right\}
$$
of all real numbers that are $\tau$-approximable by rationals whose numerator and denominator are prime ($\tau>0$). He established that $J^\prem_{\tau}$ has full (resp. zero) Lebesgue measure in $\R$ if $\tau<2$ (resp. $\tau\geq 2$). Moreover, Theorem~7 in~\cite{Durand:2007uq} implies that $J^\prem_{\tau}$ has maximal (resp. zero) Hausdorff $h$-measure in every open subset of $\R$ for any $h\in\jauge_{1}$ with $\sum_{q} h(q^{-\tau})q/(\log q)^2=\infty$ (resp. $<\infty$). In addition, $J^\prem_{\tau}$ enjoys a large intersection property, since it belongs to $\grint^h(\R)$ for any $h\in\jauge_{1}$ such that the preceding sum diverges. More generally, the results of~\cite{Durand:2007uq,Harman:1998fk} enable to describe the size and large intersection properties of the sets that are obtained when the numerator and the denominator of the rational approximates are required to belong to various subsets of $\Z$. This follows from the fact that these restricted rationals form a homogeneous ubiquitous system, see~\cite{Durand:2007uq}.

J.~Barral and S.~Seuret~\cite{Barral:2004ae,Barral:2006fk} suggested to impose new restrictions on the rational approximates: the {\em Besicovitch conditions}. Consider an integer $c\geq 2$. Each real number $x$ can be written on the form $x=x_{0}+\sum_{p} x_{p}c^{-p}$, where $x_{0}\in\Z$ and $x_{p}$ ($p\in\N$) belongs to $\intn{0}{c-1}$ and is not asymptotically constantly equal to $c-1$. For every $b\in\intn{0}{c-1}$ and every $j\in\N$, let
$$
\sigma_{s,j}(x)=\frac{1}{j}\#\left\{ p\in\intn{1}{j} \:|\: x_{p}=s \right\}.
$$
Given a probability vector $\pi=(\pi_{0},\ldots,\pi_{c-1})\in (0,1)^c$ (with $\sum_{b}\pi_{b}=1$), Besicovitch~\cite{Besicovitch:1934fk} and Eggleston~\cite{Eggleston:1949kx} investigated the set of all real numbers $x$ such that $\sigma_{b,j}(x)\to\pi_{b}$ for all $b\in\intn{0}{c-1}$ and established that its Hausdorff dimension is $\alpha=-\sum_{b} \pi_{b}\log_{c}\pi_{b}$. Thus, following the terminology of J.~Barral and S.~Seuret~\cite{Barral:2006fk}, an infinite set of rationals is said to fulfill the Besicovitch condition associated with $\pi$ if it can be enumerated as a sequence $(p_{n}/q_{n})_{n\in\N}$ enjoying $q_{n}\to\infty$ and $\sigma_{b,\lfloor 2\log_{c}q_{n} \rfloor}(p_{n}/q_{n}) \to \pi_{b}$ for all $b\in\intn{0}{c-1}$. Then, for $\tau\geq 2$, let
$$
J^\pi_{\tau}=\left\{ x\in\R \:\Biggl|\: \begin{array}{l}
\exists (p_{n}/q_{n})_{n\in\N}\text{ irreducible} \:|\: q_{n}\to\infty \\
\forall n \qquad |x-p_{n}/q_{n}|\leq {q_{n}}^{-\tau} \\
\forall b \qquad \sigma_{b,\lfloor 2\log_{c}q_{n} \rfloor}(p_{n}/q_{n}) \to \pi_{b}
\end{array}\right\}
$$
denote the subset of $J_{\tau}$ composed of all reals that are $\tau$-approximable by rationals fulfilling the Besicovitch condition associated with $\pi$. This condition makes it awkward to determine whether the rational approximates form a homogeneous ubiquitous system. Moreover, even in this case, the results of~\cite{Durand:2007uq} may not yield optimal information as regards the size and large intersection properties of $J^\pi_{\tau}$. To cope with these difficulties, J.~Barral and S.~Seuret~\cite{Barral:2004ae} introduced the notion of {\em heterogeneous ubiquitous system} (see Section~\ref{grintheterubiq}), thereby proving that the Hausdorff dimension of $J^\pi_{\tau}$ is $2\alpha/\tau$. Corollary~\ref{grintbarralseuret} below will additionally show that $J^\pi_{\tau}$ enjoys a large intersection property. Specifically, for any $\eta\in (0,1/8)$, it contains a set of the class $\grint^{g_{\alpha,\tau,\eta}}(\R)$ where $g_{\alpha,\tau,\eta}(r)=r^{\frac{2\alpha}{\tau}-3(-\log r)^{\eta-1/8}}$. As a result, it also contains a set of the class $\grintfalc^{2\alpha/\tau}(\R)$ of K.~Falconer.

A neighboring problem is to describe the size and large intersection properties of the set
$$
\tilde J^\pi_{\tau}=\left\{ x\in\R \:\Biggl|\: \begin{array}{l}
\exists (p_{n}/q_{n})_{n\in\N}\text{ irreducible} \:|\: q_{n}\to\infty \\
\forall n \qquad |x-p_{n}/q_{n}|\leq {q_{n}}^{-\tau} \\
\forall b \qquad \sigma_{b,\lfloor 2\log_{c}q_{n} \rfloor}(x) \to \pi_{b}
\end{array}\right\}.
$$
The Besicovitch condition now bears on the approximated reals rather than on the rational approximates. J.~Barral and S.~Seuret~\cite{Barral:2006fk} established that some subset of $\tilde J^\pi_{\tau}$ belongs to the class $\grintfalc^{2\alpha/\tau}(\R)$ of K.~Falconer. A variant of Corollary~\ref{grintbarralseuret} enables to refine this result: $\tilde J^\pi_{\tau}$ contains a set of the class $\grint^{g_{\alpha,\tau,\eta}}(\R)$, which is strictly included in $\grintfalc^{2\alpha/\tau}(\R)$. The methods of~\cite{Durand:2007uq} cannot be applied here. Indeed, $\tilde J^\pi_{\tau}$ is included in the set of all reals $x$ such that $\sigma_{b,\lfloor 2\log_{c}q_{n} \rfloor}(x) \to \pi_{b}$ for all $b\in\intn{0}{c-1}$ and some sequence $(q_{n})_{n\in\N}$ diverging to infinity. As shown in~\cite{Barral:2006fk}, the dimension of this last set is $\alpha$. Thus, $\tilde J^\pi_{\tau}$ has Lebesgue measure zero for every $\tau>0$ if $\alpha<1$, so that the rational approximates do not form a homogeneous ubiquitous system.

The asymptotic frequencies of the digits of real numbers are linked with the local behavior of the {\em multinomial measures}, so that the Besicovitch conditions can be recast using these measures, see Section~\ref{grintheterapplic}. Consequently, the study of the large intersection properties of $J^\pi_{\tau}$ is a special case of the following general problem. Let us endow $\R^d$ with the supremum norm, let $I$ denote a denumerable set and let $(x_{i},r_{i})_{i\in I}$ be a family in $[0,1]^d\times (0,\infty)$ such that zero is the only cluster point of $(r_{i})_{i\in I}$. The size and large intersection properties of the set
$$
F_{t}=\left\{ x\in\R^d \:\bigl|\: \|x-k-x_{i}\|<{r_{i}}^t\text{ for infinitely many }(i,k)\in I\times\Z^d\right\},
$$
where $t\in [1,\infty)$, are well determined when the family $(k+x_{i},r_{i})_{(i,k)\in I\times\Z^d}$ forms a homogeneous ubiquitous system. To be specific, Theorem~2 in~\cite{Durand:2007uq} ensures that $F_{t}\in\grint^{r^{d/t}}(\R^d)$. Note that $F_{t}$ is the natural generalization of the set $J_{\tau}$ of all reals that are $\tau$-approximable by rationals. Likewise, the set $J^\pi_{\tau}\subseteq J_{\tau}$ of all reals that are $\tau$-approximable by rationals fulfilling the Besicovitch condition associated with $\pi$ is comparable with
$$
E_{t}=\left\{ x\in\R^d \:\bigl|\: \|x-k-x_{i}\|<{r_{i}}^t\text{ for infinitely many }(i,k)\in I^{\mu,\alpha}\times\Z^d\right\}\subseteq F_{t},
$$
where $I^{\mu,\alpha}$ denotes the set of all $i\in I$ such that the $\mu$-mass of the ball with center $x_{i}$ and radius $r_{i}$ behaves as ${r_{i}}^\alpha$ for some Borel measure $\mu$ and some $\alpha>0$. J.~Barral and S.~Seuret~\cite{Barral:2004ae} computed the Hausdorff dimension of $E_{t}$ when $(x_{i},r_{i})_{i\in I}$ is a heterogeneous ubiquitous system with respect to $\mu$, which imposes certain conditions on the repartition of the balls with center $x_{i}$ and radius $r_{i}$ and on the local behavior of $\mu$, see Section~\ref{grintheterubiq}. As shown by Theorem~\ref{UBIQUITYHETER} in this paper, the set $E_{t}$ additionally enjoys a large intersection property, in the sense that it lies in the class $\grint^h(\R^d)$ for a certain gauge function $h\in\jauge_{d}$.

The description of the large intersection properties of $J^\pi_{\tau}$ is deduced from Theorem~\ref{UBIQUITYHETER} by picking the multinomial measure associated with $\pi$ to play the role of $\mu$. Similarly, by choosing the Gibbs measure associated with some H\"older continuous function $f$, we infer the large intersection properties of the set of all points that are approximable by rationals where the average of the Birkhoff sum associated with $f$ has a given limit, see Section~\ref{grintheterapplic}.

The paper is organized as follows. In Section~\ref{grinthetersets} we recall the definition of the class $\grint^h(V)$ of sets with large intersection in a given nonempty open set $V$ with respect to a given gauge function $h\in\jauge_{d}$, which was introduced in~\cite{Durand:2007uq}. In addition, we supply a sufficient condition expressed in terms of similarities to ensure that a set belongs to $\grint^h(\R^d)$ where $h$ is of a certain form. This condition is convenient to establish Theorem~\ref{UBIQUITYHETER} given in Section~\ref{grintheterubiq} according to which the aforementioned set $E_{t}$ enjoys a large intersection property if the family $(x_{i},r_{i})_{i\in I}$ is a heterogeneous ubiquitous system with respect to some measure $\mu$. Section~\ref{grintheterapplic} then provides several applications to metric number theory. Section~\ref{grintheterproofsets} and Section~\ref{grintheterproofubiq} are devoted to proving the main results of the paper.

\section{Sets with large intersection and similarities}\label{grinthetersets}

Recall that the set $\jauge_{d}$ of gauge functions is defined at the beginning of Section~\ref{grintheterintro}. For any $h\in\jauge_{d}$, the Hausdorff $h$-measure is given by
$$
\forall F\subseteq \R^d \qquad \hau^h(F)=\lim_{\delta\downarrow 0}\uparrow \inf_{(U_{p})_{p\in\N} \atop \diam{U_{p}}<\delta} \sum_{p=1}^\infty h(\diam{U_{p}})
$$
where the infimum is taken over all sequences $(U_{p})_{p\in\N}$ of sets with $F\subseteq\bigcup_{p}U_{p}$ and $\diam{U_{p}}<\delta$ for all $p\in\N$, where $\diam{\cdot}$ denotes diameter. This is a Borel measure on $\R^d$. The Hausdorff dimension of a nonempty set $F\subseteq\R^d$ is defined by
$$
\dim F=\sup\{s\in (0,d) \:|\: \hau^{r^s}(F)=\infty\}=\inf\{s\in (0,d) \:|\: \hau^{r^s}(F)=0\}
$$
with the convention that $\sup\emptyset=0$ and $\inf\emptyset=d$, see~\cite{Falconer:2003oj,Rogers:1970wb}.

In order to refine the classes of K.~Falconer, we introduced in~\cite{Durand:2007uq} the class $\grint^h(V)$ of sets with large intersection in a given nonempty open subset $V$ of $\R^d$ with respect to a given function $h\in\jauge_{d}$. Contrary to the classes of K.~Falconer, the class $\grint^h(V)$ is defined using outer net measures rather than similarities. Specifically, given an integer $c\geq 2$, let $\Lambda_{c}$ denote the collection of all $c$-adic cubes of $\R^d$, i.e. sets of the form $\lambda=\lambda^c_{j,k}=c^{-j}(k+[0,1)^d)$ where $j\in\Z$ and $k\in\Z^d$. The integer $j$ is called the generation of $\lambda$ and is denoted by $\gene{\lambda}_{c}$. The outer net measure associated with $h$ is given by
$$
\forall F\subseteq\R^d \qquad \netm^h_{\infty}(F) = \inf_{(\lambda_{p})_{p\in\N}\in R_{c,h}(F)} \sum_{p=1}^\infty h(\diam{\lambda_{p}})
$$
where $R_{c,h}(F)$ is the set of all sequences $(\lambda_{p})_{p\in\N}$ in $\Lambda_{c}\cup\{\emptyset\}$ such that $F$ is included in $\bigcup_{p}\lambda_{p}$ and $\diam{\lambda_{p}}$ is less than $\eps_{h}$, which is the supremum of all $\eps\in (0,1]$ such that $h$ is continuous and nondecreasing on $[0,\eps]$ and $r\mapsto h(r)/r^d$ is positive and nonincreasing on $(0,\eps]$. Moreover, recall that a $G_{\delta}$-set is one that may be expressed as a countable intersection of open sets. The class $\grint^h(V)$ is then defined as follows.

\begin{df}
Let $h\in\jauge_{d}$ and let $V$ be a nonempty open subset of $\R^d$. The class $\grint^h(V)$ of subsets of $\R^d$ with large intersection in $V$ with respect to $h$ is the collection of all $G_{\delta}$-subsets $F$ of $\R^d$ such that $\netm^g_{\infty}(F\cap U)=\netm^g_{\infty}(U)$ for every $g\in\jauge_{d}$ enjoying $g\prec h$ and every open set $U\subseteq V$.
\end{df}

The class $\grint^h(V)$ enjoys the same kind of stability properties as the classes of K.~Falconer, as shown by the following result of~\cite{Durand:2007uq}. It is also proven in the same paper that $\grint^h(V)$ depends on the choice of neither the integer $c$ nor the norm $\R^d$ is endowed with.

\begin{thm}\label{GRINTSTABLE}
Let $h\in\jauge_{d}$ and let $V$ be a nonempty open subset of $\R^d$. Then
\begin{itemize}
\item the class $\grint^h(V)$ is closed under countable intersections;
\item the set $f^{-1}(F)$ belongs to $\grint^h(V)$ for every bi-Lipschitz mapping $f:V\to\R^d$ and every set $F\in\grint^h(f(V))$;
\item every set $F\in\grint^h(V)$ enjoys $\hau^g(F)=\infty$ for every $g\in\jauge_{d}$ with $g\prec h$ and in particular $\dim F\geq s_{h}=\sup\left\{s\in (0,d) \:|\: r^s\prec h \right\}$.
\end{itemize}
\end{thm}

By Theorem~\ref{GRINTSTABLE}, every sequence $(F_{n})_{n\in\N}$ in $\grint^h(V)$ enjoys $\hau^g(\bigcap_{n}F_{n})=\infty$ for every $g\in\jauge_{d}$ with $g\prec h$. In addition, $\grint^h(\R^d)$ is included in the class $\grintfalc^{s_{h}}(\R^d)$ of K.~Falconer if $s_{h}$ is positive. More precisely, each $F\in\grint^h(\R^d)$ satisfies $\hau^g(\bigcap_{n} f_{n}(F))=\infty$ for every $g\in\jauge_{d}$ with $g\prec h$ and every sequence $(f_{n})_{n\in\N}$ of similarities. Proposition~\ref{HAUSIMDCGRINT} below provides a partial converse to this result. In its statement, $D_{c}$ denotes the set of all dilations that map a $c$-adic cube $\lambda$ with diameter less than $1$ to a $c$-adic cube with generation greater than or equal to that of $\lambda$. Moreover, $\Phi$ is the set of all functions $\ph$ that are continuous and nondecreasing on $[0,\rho]$ for some $\rho>0$, vanish at zero and are such that $r^{-\ph(r)}$ monotonically tends to infinity as $r\to 0$ and such that $r\mapsto r^{\eps-\ph(r)}$ tends to zero at zero and increases in $(0,\rho_{\eps}]$ for some $\rho_{\eps}>0$ and any $\eps>0$. Thus, for every $\beta\in (0,d]$ and every $\ph\in\Phi$, the function $h_{\beta,\ph}:r\mapsto r^{\beta-\ph(r)}$ belongs to $\jauge_{d}$ and the class $\grint^{h_{\beta,\ph}}(\R^d)$ is strictly contained in $\grintfalc^{\beta}(\R^d)$.

\begin{prp}\label{HAUSIMDCGRINT}
Let $\beta\in (0,d]$, let $\ph\in\Phi$ and let $F$ be a $G_{\delta}$-subset of $\R^d$ such that
\begin{equation}\label{hauhcapfnfpos}
\hau^{h_{\beta,\ph}}\left( \bigcap_{n=1}^\infty f_{n}(F) \right)>0
\end{equation}
for every sequence $(f_{n})_{n\in\N}$ in $D_{c}$. Then $F$ belongs to $\grint^{h_{\beta,\ph}}(\R^d)$.
\end{prp}

We refer to Section~\ref{grintheterproofsets} for a proof of this result. The gauge functions $h_{\beta,\ph}$, for $\beta\in (0,d]$ and $\ph\in\Phi$, are precisely those which arise in the study of the large intersection properties of the sets introduced by J.~Barral and S.~Seuret in~\cite{Barral:2004ae}, see Section~\ref{grintheterubiq}. As an example of functions of the form $h_{\beta,\ph}$, let us mention
$$
r\mapsto r^{\beta}\prod_{p=1}^\infty \left(\log^{\circ p}\frac{1}{r}\right)^{\nu_{p}}
$$
where $\log^{\circ p}$ stands for the $p$-th iterate of the logarithm, $\beta\in (0,d]$ and $(\nu_{p})_{p\in\N}$ denotes a sequence in $[0,\infty)$ such that all $\nu_{p}$ vanish except a finite number of them.

\section{Heterogeneous ubiquity}\label{grintheterubiq}

We begin by recalling the notion of homogeneous ubiquitous system introduced in~\cite{Durand:2007uq}. Given a denumerable set $I$, let $\S^0_{d}(I)$ denote the set of all families $(x_{i},r_{i})_{i\in I}$ in $[0,1]^d\times (0,\infty)$ such that, for every $\eps>0$, there are at most finitely many $i\in I$ enjoying $r_{i}>\eps$. A family $(x_{i},r_{i})_{i\in I}\in\S^0_{d}(I)$ is a homogeneous ubiquitous system in $(0,1)^d$ if for Lebesgue-almost every $x\in [0,1]^d$, there are infinitely many $i\in I$ enjoying $\|x-x_{i}\|\leq r_{i}$. Examples of homogeneous ubiquitous systems include the rational numbers, the real algebraic numbers of bounded degree, the algebraic integers of bounded degree and, more generally, the {\em optimal regular systems} which are very common in the theory of Diophantine approximation, see~\cite{Durand:2007uq}.

The notion of heterogeneous ubiquitous system was introduced by J.~Barral and S.~Seuret in~\cite{Barral:2004ae}. Let us endow $\R^d$ with the supremum norm. To be a heterogeneous ubiquitous system, a family $(x_{i},r_{i})_{i\in I}\in\S^0_{d}(I)$ must satisfy four assertions which we detail now. Let $\mesb$ be the set of all finite Borel measures with support $[0,1]^d$ and let $m\in\mesb$. The first assertion is:\renewcommand{\theenumi}{\roman{enumi}}
\begin{enumerate}
\item\label{condcov} For $m$-almost every $x\in [0,1]^d$, there are infinitely many $i\in I$ such that $\|x-x_{i}\|\leq r_{i}/2$.
\end{enumerate}

Let us mention some examples of families for which~(\ref{condcov}) holds. Given an integer $c\geq 2$, let $\I_{d,c}$ be the set of all pairs $(j,k)$ with $j\in\N$ and $k\in\intn{0}{c^j}^d$. For each $x\in [0,1]^d$ there are infinitely many $(j,k)\in\I_{d,c}$ such that $\|x-kc^{-j}\|\leq c^{-j}/2$. Thus,~(\ref{condcov}) holds for the family $(kc^{-j},c^{-j})_{(j,k)\in\I_{d,c}}\in\S^0_{d}(\I_{d,c})$ and for any $m\in\mesb$.

Likewise, let $\I_{d,{\rm rat}}$ be the set of all pairs $(p,q)$ with $q\in\N$ and $p\in\intn{0}{q}^d$. By Dirichlet's theorem,~(\ref{condcov}) holds for the family $(p/q,2q^{-1-1/d})_{(p,q)\in\I_{d,{\rm rat}}}\in\S^0_{d}(\I_{d,{\rm rat}})$ along with any $m\in\mesb$, see~\cite[Theorem~200]{Hardy:1979fk}. Furthermore, given $x\in\R^d\backslash\Q^d$, let $\kappa(x)$ be the infimum of all $\kappa\in (0,1]$ such that $\|x-p/q\|<\kappa^{1/d}q^{-1-1/d}$ holds for infinitely many $(p,q)\in\Z^d\times\N$ where $p/q$ has at least one irreducible coordinate and let $\gamma_{d}$ be the supremum of those $\kappa(x)$. Hurwitz~\cite{Hurwitz:1891uq} showed that $\gamma_{1}=1/\sqrt{5}$, see~\cite[Theorem~193]{Hardy:1979fk}. Let $\I_{1,{\rm rat}}^*$ be the set of all pairs $(p,q)$ with $q\in\N$ and $p\in\intn{1}{q-1}$ such that $p/q$ is irreducible. As $\gamma_{1}<1/2$,~(\ref{condcov}) holds in dimension $1$ for the family $(p/q,1/q^2)_{(p,q)\in\I_{1,{\rm rat}}^*}\in\S^0_{1}(\I_{1,{\rm rat}}^*)$ along with any $m\in\mesb$ enjoying $m(\Q)=0$. We shall use this result in order to prove Corollary~\ref{grintbarralseuret}, see Section~\ref{grintheterapplic}. Note that various bounds on $\gamma_{d}$ have been established in the case where $d\geq 2$ but its exact value is still unknown, see~\cite[Section~2.23]{Finch:2003kx}.

Let us define the three other assertions. Consider an integer $c\geq 2$. Given $x\in\R^d$ and $j\in\Z$, let $\lambda^c_{j}(x)$ denote the unique $c$-adic cube with generation $j$ that contains $x$. Moreover, write $3\lambda$ for the cube obtained by expanding any $\lambda\in\Lambda_{c}$ by a factor $3$ about its center. The second assertion states that the local behavior of a given measure $m\in\mesb$ is controlled $m$-almost everywhere by the function $h_{\beta,\ph}\in\jauge_{d}$ for some $\beta\in (0,d]$ and some $\ph\in\Phi$:\begin{enumerate}\setcounter{enumi}{1}
\item\label{condm} For $m$-almost every $x\in [0,1]^d$, there is an integer $j(x)$ such that, for every integer $j\geq j(x)$ and every $k\in\intn{0}{c^j-1}^d$,
$$
\lambda^c_{j,k} \subseteq 3\lambda^c_{j}(x) \quad\Longrightarrow\quad m(\lambda^c_{j,k})\leq h_{\beta,\ph}(\diam{\lambda^c_{j,k}}).
$$
\end{enumerate}

Let $\Psi$ be the set of all functions $\psi$ that are continuous and nondecreasing on $[0,\rho]$ for some $\rho>0$, vanish at zero and are such that $r\mapsto r^{-\psi(r)}$ is nonincreasing on $(0,\rho]$. The third assertion indicates that a measure $m\in\mesb$ focuses on the points where a given measure $\mu\in\mesb$ follows a power-law behavior with exponent $\alpha>0$ up to a correction $\psi\in\Psi$:\begin{enumerate}\setcounter{enumi}{2}
\item\label{condmu} For $m$-almost every $x\in [0,1]^d$, there is an integer $j(x)$ such that, for every integer $j\geq j(x)$ and every $k\in\intn{0}{c^j-1}^d$,
$$
\lambda^c_{j,k} \subseteq 3\lambda^c_{j}(x) \quad\Longrightarrow\quad \diam{\lambda^c_{j,k}}^{\alpha+\psi(\diam{\lambda^c_{j,k}})} \leq \mu(\lambda^c_{j,k}) \leq \diam{\lambda^c_{j,k}}^{\alpha-\psi(\diam{\lambda^c_{j,k}})}.
$$
\end{enumerate}

The last assertion imposes a self-similarity condition on a measure $m\in\mesb$. To be specific, given a $c$-adic cube $\lambda$ with nonnegative generation, let $\omega_{\lambda}$ denote the dilation that maps $\lambda$ to the unique cube of vanishing generation that contains it. Moreover, for any $c$-adic cube $\lambda\subseteq [0,1)^d$, let $m^\lambda=m\circ\omega_{\lambda}$. The last assertion is:
\begin{enumerate}\setcounter{enumi}{3}
\item\label{condselfsim} For every $c$-adic cube $\lambda\subseteq [0,1)^d$, the measure $m^\lambda$ and the restriction $m\restr{\lambda}$ of $m$ to $\lambda$ are absolutely continuous with respect to one another.
\end{enumerate}\renewcommand{\theenumi}{\arabic{enumi}}

Two important examples of measures for which~(\ref{condm})-(\ref{condselfsim}) hold are discussed in Section~\ref{grintheterapplic}: the products of multinomial measures and the Gibbs measures associated with a H\"older continuous function.

\begin{df}
Let $I$ be a denumerable set. Any family $(x_{i},r_{i})_{i\in I}\in\S^0_{d}(I)$ is called a heterogeneous ubiquitous system with respect to $(\mu,\alpha,\beta,\ph)\in\mesb\times (0,\infty)\times (0,d]\times\Phi$ if~(\ref{condcov})-(\ref{condselfsim}) hold for some $\psi\in\Psi$, some $m\in\mesb$ and some integer $c\geq 2$.
\end{df}

The homogeneous ubiquitous systems of~\cite{Durand:2007uq} are a particular case of the heterogeneous ones. Indeed, given a denumerable set $I$, consider a family $(x_{i},r_{i})_{i\in I}\in\S^0_{d}(I)$ which is a homogeneous ubiquitous system in $(0,1)^d$. Proposition~15 in~\cite{Durand:2007uq} implies that~(\ref{condcov}) holds when $m$ is the Lebesgue measure on $[0,1]^d$, which is denoted by $\leb^d$. It is then straightforward to establish that $(x_{i},r_{i})_{i\in I}$ is a heterogeneous ubiquitous system with respect to $(\leb^d,d,d,\ph)$ for every $\ph\in\Phi$.

Let $I$ denote a denumerable set and let $(x_{i},r_{i})_{i\in I}$ be a heterogeneous ubiquitous system with respect to $(\mu,\alpha,\beta,\ph)\in\mesb\times (0,\infty)\times (0,d]\times\Phi$. Thus,~(\ref{condcov})-(\ref{condselfsim}) hold for some $\psi\in\Psi$, some $m\in\mesb$ and some integer $c\geq 2$. Moreover, for any $M\in (0,\infty)$, let $I^{\mu,\alpha}_{M,\psi}$ be the set of all $i\in I$ such that
\begin{equation}\label{ballsimcond}
M^{-1} (2r_{i})^{\alpha+\psi(2r_{i})} \leq \mu\left( B(x_{i},r_{i}) \right) \leq \mu\left( \bar B(x_{i},r_{i}) \right) \leq M (2r_{i})^{\alpha-\psi(2r_{i})}
\end{equation}
where $B(x_{i},r_{i})$ and $\bar B(x_{i},r_{i})$ are the open and the closed balls with center $x_{i}$ and radius $r_{i}$ respectively. Given $t\in [1,\infty)$, J.~Barral and S.~Seuret investigated the size properties of the set
\begin{equation}\label{tildeetlimsupballsim}
\tilde E_{t}=\left\{ x\in [0,1]^d \:\bigl|\: \|x-x_{i}\|<{r_{i}}^t\text{ for infinitely many }i\in I^{\mu,\alpha}_{M,\psi} \right\}
\end{equation}
of all points in $[0,1]^d$ that lie infinitely often in an open ball $B(x_{i},{r_{i}}^t)$ when $i$ is such that the $\mu$-mass of $B(x_{i},r_{i})$ behaves as ${r_{i}}^\alpha$. Specifically, they established that for some $M\in (0,\infty)$ and every $t\in [1,\infty)$, the Hausdorff dimension of $\tilde E_{t}$ is at least $\beta/t$, see~\cite[Theorem~2.7]{Barral:2004ae}. Thus, the same property holds for
\begin{equation}\label{etlimsupballsim}\begin{split}
E_{t} &= \Z^d+\tilde E_{t} \\
&= \left\{ x\in \R^d \:\bigl|\: \|x-p-x_{i}\|<{r_{i}}^t\text{ for infinitely many }(i,p)\in I^{\mu,\alpha}_{M,\psi}\times\Z^d \right\}.
\end{split}\end{equation}
On top of that, the sets $\tilde E_{t}$ and $E_{t}$ enjoy a large intersection property, as shown by the following theorem which is proven in Section~\ref{grintheterproofubiq}.

\begin{thm}\label{UBIQUITYHETER}
Let $I$ be a denumerable set and let $(x_{i},r_{i})_{i\in I}$ be a heterogeneous ubiquitous system with respect to $(\mu,\alpha,\beta,\ph)\in\mesb\times (0,\infty)\times (0,d]\times\Phi$. Then, for some $\psi\in\Psi$ and some $M\in [1,\infty)$:\begin{itemize}
\item $\tilde E_{1}\in\grint^{h_{\beta,\ph}}((0,1)^d)$ and $E_{1}\in\grint^{h_{\beta,\ph}}(\R^d)$;
\item $\tilde E_{t}\in\grint^{h_{\beta/t,2\ph+\phi}}((0,1)^d)$ and $E_{t}\in\grint^{h_{\beta/t,2\ph+\phi}}(\R^d)$ for all $t\geq 1$ and $\phi\in\Phi$.
\end{itemize}
\end{thm}

\begin{rem}
This result is reminiscent of Theorem~2 in~\cite{Durand:2007uq} which discusses the large intersection properties of the sets built on homogeneous ubiquitous systems. Note that it is pointless to apply the previous result rather than Theorem~2 in~\cite{Durand:2007uq} to a homogeneous ubiquitous system (which is also a heterogeneous ubiquitous system). Indeed, let $(x_{i},r_{i})_{i\in I}\in\S^0_{d}(I)$ denote a homogeneous ubiquitous system in $(0,1)^d$. Then $(x_{i},r_{i})_{i\in I}$ is a heterogeneous ubiquitous system with respect to $(\leb^d,d,d,\ph)$ for all $\ph\in\Phi$. Furthermore, $I^{\leb^d,d}_{M,\psi}=I$ for all $\psi\in\Psi$ and $M\in [1,\infty)$. Thus, for all $t\in [1,\infty)$, Theorem~2 in~\cite{Durand:2007uq} ensures that $\tilde E_{t}\in\grint^{r^{d/t}}((0,1)^d)$, while Theorem~\ref{UBIQUITYHETER} above implies that $\tilde E_{t}\in\grint^{r^{d/t-\ph(r)}}((0,1)^d)$ for all $\ph\in\Phi$, which is a weaker result since some gauge functions $h\prec r^{d/t}$ are not of the form $h(r)=r^{d/t-\ph(r)}$.
\end{rem}

Theorem~\ref{UBIQUITYHETER} enables to investigate the size properties of the intersection of a countable number of sets built on heterogeneous ubiquitous systems. Let $J$ denote a countable set and, for each $j\in J$, let $(x^j_{i},r^j_{i})_{i\in I_{j}}$ (where $I_{j}$ is a denumerable set) be a heterogeneous ubiquitous system with respect to $(\mu^j,\alpha^j,\beta^j,\ph^j)\in\mesb\times (0,\infty)\times (0,d]\times\Phi$. Assume that $\inf_{j} \beta^j/t^j>0$. Then, by Theorem~\ref{GRINTSTABLE} and Theorem~\ref{UBIQUITYHETER}, there are a family $(M^j)_{j\in J}$ in $[1,\infty)$ and a family $(\psi^j)_{j\in J}$ in $\Psi$ such that
$$
\bigcap_{j\in J} \tilde E^j_{t^j} \in \bigcap_{h\in\jauge_{d} \atop \forall j, h\prec h_{\beta^j/t^j,2\ph^j+\phi^j}} \grint^h((0,1)^d)
$$
for every family $(\phi^j)_{j\in J}$ in $\Phi$ and every family $(t^j)_{j\in J}$ in $[1,\infty)$, where $\tilde E^j_{t^j}$ is defined as in~(\ref{tildeetlimsupballsim}). As a consequence, the Hausdorff dimension of $\bigcap_{j} \tilde E^j_{t^j}$ is at least $\inf_{j} \beta^j/t^j$. This result can be seen as the analog of Theorem~2 in~\cite{Dodson:1990fc} for sets built on heterogeneous ubiquitous systems.

\section{Applications to metric number theory}\label{grintheterapplic}

In~\cite{Barral:2004ae}, J.~Barral and S.~Seuret gave several examples of heterogeneous ubiquitous systems $(x_{i},r_{i})_{i\in I}\in\S^0_{d}(I)$ (where $I$ is a denumerable set) with respect to some tuple $(\mu,\alpha,\beta,\ph)\in\mesb\times (0,\infty)\times (0,d]\times\Phi$. We review them in this section and we show that the corresponding sets $\tilde E_{t}$ and $E_{t}$ which are defined in Section~\ref{grintheterubiq} enjoy a large intersection property.

\subsection{Products of multinomial measures}

To begin with, let us recall the definition of multinomial measures. Consider an integer $c\geq 2$. Any $c$-adic interval $\lambda=c^{-j}(k+[0,1))\subset [0,1)$ ($j\in\N$, $k\in\intn{0}{c^j-1}$) can be coded by some $u(\lambda)=(u(\lambda)_{1},\ldots,u(\lambda)_{j})$, where $u(\lambda)_{p}=\lfloor k c^{p-j}\rfloor \modu c$ for all $p\in\intn{1}{j}$. Let $\pi=(\pi_{0},\ldots,\pi_{c-1})\in (0,1)^c$ with $\sum_{b}\pi_{b}=1$. Then $\pi$ is called a probability vector. For each $j\in\N$, let
$$
\mu^\pi_{j}({\rm d}x)=\sum_{\lambda\subset [0,1) \atop \gene{\lambda}_{c}=j} c^j \ind_{\lambda}(x) \prod_{p=1}^j \pi_{u(\lambda)_{p}} {\rm d}x
$$
where the sum is over all $c$-adic subintervals of $[0,1)$ with generation $j$. Then $(\mu^\pi_{j})_{j\in\N}$ converges weakly to some Borel probability measure $\mu_{\pi}$ with support $[0,1]$. This measure is called the multinomial measure associated with $\pi$. Its multifractal analysis (that is, the computation of the Hausdorff dimension of the set $V_{\alpha}$, $\alpha\in\R$, of all real numbers $x\in [0,1)$ such that $-\frac{1}{j}\log_{c}\mu_{\pi}(\lambda^c_{j}(x)) \to \alpha$) involves
$$
\forall q\in\R \qquad \tau_{\mu_{\pi}}(q)=-\lim_{j\to\infty} \frac{1}{j}\log_{c}\sum_{\lambda\subset [0,1) \atop \gene{\lambda}_{c}=j} \left( \mu_{\pi}(\lambda) \right)^q=-\log_{c}\sum_{b=0}^{c-1} {\pi_{b}}^q.
$$
Let $q\in\R$, $\alpha=\tau_{\mu_{\pi}}'(q)$ and $\beta=q\tau_{\mu_{\pi}}'(q)-\tau_{\mu_{\pi}}(q)$. Then, studying the multinomial measure $\mu_{\pi,q}$ associated with $c^{\tau_{\mu_{\pi}}(q)}({\pi_{0}}^q,\ldots,{\pi_{c-1}}^q)$, one can show that the Hausdorff dimension of $V_{\alpha}$ is $\beta$, see~\cite{Brown:1992fk}.

Let us turn our attention to products of multinomial measures and investigate the large intersection properties of the corresponding sets $\tilde E_{t}$ and $E_{t}$ which are defined in Section~\ref{grintheterubiq}. Specifically, let $\pi^1,\ldots,\pi^d$ be probability vectors and let $\mu=\mu_{\pi^{1}}\otimes\cdots\otimes\mu_{\pi^d}$. Then $\mu$ is a Borel probability measure with support $[0,1]^d$ and its multifractal analysis involves $\tau_{\mu}=\tau_{\mu_{\pi^1}}+\cdots+\tau_{\mu_{\pi^d}}$. Given $q\in\R$, let $\alpha=\tau_{\mu}'(q)$ and $\beta=q\tau_{\mu}'(q)-\tau_{\mu}(q)$. Moreover, let $m=\mu_{\pi^{1},q}\otimes\cdots\otimes\mu_{\pi^d,q}$, $\ph=f_{\eta}:r\mapsto \left(-\log r\right)^{\eta-1/8}$ and $\psi=f_{\eta'}$, where $\eta,\eta'\in (0,1/8)$. Then~(\ref{condm})-(\ref{condselfsim}) hold, by Theorem~1 in~\cite{Barral:2005uq}. Suppose that~(\ref{condcov}) holds for some family $(x_{i},r_{i})_{i\in I}\in\S^0_{d}(I)$. Thus, this family is a heterogeneous ubiquitous system with respect to $(\mu,\alpha,\beta,\ph)$ and Theorem~\ref{UBIQUITYHETER} ensures that $\tilde E_{t}\in\grint^{h_{\beta/t,3f_{\eta}}}((0,1)^d)$ and $E_{t}\in\grint^{h_{\beta/t,3f_{\eta}}}(\R^d)$ for some $M\in [1,\infty)$ and all $t\in [1,\infty)$.

We now exploit the link between multinomial measures and digits of points in order to show that the points that are approximated by rationals which satisfy a given Besicovitch condition form a set with large intersection. Let $q=1$ in what precedes, so that $\alpha=\beta=\tau_{\mu}'(1)$ and $m=\mu$. Each real number $x$ can be written on the form $x=x_{0}+\sum_{p=1}^\infty x_{p}c^{-p}$ where $x_{0}\in\Z$ and $x_{p}$ ($p\in\N$) belongs to $\intn{0}{c-1}$ and is not asymptotically constantly equal to $c-1$. Consider
$$
\forall b\in\intn{0}{c-1} \quad \forall j\in\N \qquad \sigma_{b,j}(x)=\frac{1}{j}\#\left\{ p\in\intn{1}{j} \:|\: x_{p}=b \right\}.
$$
The law of large numbers implies that the set of all points $x=(x^1,\ldots,x^d)\in [0,1]^d$ such that $\sigma_{b,j}(x^s)\to\pi^s_{b}$ for all $s\in\intn{1}{d}$ and $b\in\intn{0}{c-1}$ has full $\mu$-measure. As a result, we deduce that its Hausdorff dimension is at least
\begin{equation}\label{dimeggleston}
\alpha=\tau_{\mu}'(1)=-\sum_{s=1}^d\sum_{b=0}^{c-1} \pi^s_{b}\log_{c}\pi^s_{b},
\end{equation}
thus recovering a well-known result of Besicovitch~\cite{Besicovitch:1934fk} and Eggleston~\cite{Eggleston:1949kx}. The law of the iterated logarithm (see e.g.~\cite{Rogers:1994ug}) yields a more precise result. Specifically, for $m$-almost every $x\in [0,1]^d$ there is an integer $j(x)\in\N$ such that
$$
\sup_{z\in [0,1]^d\cap 3\lambda^c_{j}(x)} \max_{s\in\intn{1}{d} \atop b\in\intn{0}{c-1}} \left|\sigma_{b,j}(z^s)-\pi^s_{b}\right|\leq \sqrt{\log c}\, \varsigma(c^{-j})
$$
for all integer $j\geq j(x)$, where
\begin{equation}\label{correcvarsigma}
\varsigma:r\mapsto\left(\frac{\log\log\log \frac{1}{r}}{\log \frac{1}{r}}\right)^{1/2}.
\end{equation}
Let $I^\devel_{\pi}$ denote the set of all $i\in I$ such that
\begin{equation}\label{develpicond}
\max_{s\in\intn{1}{d} \atop b\in\intn{0}{c-1}} \left|\sigma_{b,\lfloor -\log_{c}r_{i} \rfloor}(x_{i}^s)-\pi^s_{b}\right|\leq 2\sqrt{\log c}\,\varsigma(r_{i}).
\end{equation}
The following result is established in Section~\ref{grintheterproofubiq}.

\begin{prp}\label{ubiquitymultinom}
Consider $d$ probability vectors $\pi^1,\ldots,\pi^d\in (0,1)^c$ ($c\geq 2$). Let $I$ be a denumerable set and let $(x_{i},r_{i})_{i\in I}\in\S^0_{d}(I)$. Assume that~(\ref{condcov}) holds for $m=\mu_{\pi^{1}}\otimes\cdots\otimes\mu_{\pi^d}$. Then, the set
$$
\tilde E^\devel_{t}=\left\{ x\in [0,1]^d \:\bigl|\: \|x-x_{i}\|<{r_{i}}^t\text{ for infinitely many }i\in I^\devel_{\pi} \right\}
$$
belongs to $\grint^{h_{\alpha/t,3f_{\eta}}}((0,1)^d)$ and $E^\devel_{t}=\Z^d+\tilde E^\devel_{t}$ belongs to $\grint^{h_{\alpha/t,3f_{\eta}}}(\R^d)$ for all $t\in [1,\infty)$ and $\eta\in (0,1/8)$, where $\alpha$ is given by~(\ref{dimeggleston}) and $f_{\eta}:r\mapsto (-\log r)^{\eta-1/8}$.
\end{prp}

Proposition~\ref{ubiquitymultinom} enables to determine the large intersection properties of the set
\begin{equation}\label{defupit}
U_{\pi,t}=\left\{ x\in\R \:\Biggl|\: \begin{array}{l}
\exists (p_{n}/q_{n})_{n\in\N}\text{ irreducible} \:|\: q_{n}\to\infty \\
\forall n \qquad |x-p_{n}/q_{n}|\leq {q_{n}}^{-2t} \\
\forall b \qquad \sigma_{b,\lfloor 2\log_{c}q_{n} \rfloor}(p_{n}/q_{n}) \to \pi_{b}
\end{array}\right\}
\end{equation}
of all real numbers that are $2t$-approximable ($t\geq 1$) by rationals which satisfy the Besicovitch condition associated with a given probability vector $\pi=(\pi_{0},\ldots,\pi_{c-1})\in (0,1)^c$. Indeed, let $I=\I_{1,{\rm rat}}^*$ (see Section~\ref{grintheterubiq}) and let $x_{(p,q)}=p/q$ and $r_{(p,q)}=1/q^{2}$ for every $(p,q)\in I$. Then~(\ref{condcov}) holds for $m=\mu_{\pi}$ since $\mu_{\pi}(\Q)=0$. In addition, $U_{\pi,t}\supseteq E^\devel_{t}$ because each injective sequence $(p_{n},q_{n})_{n\in\N}$ in $I^\devel_{\pi}$ enjoys $\sigma_{b,\lfloor 2\log_{c}q_{n} \rfloor}(p_{n}/q_{n})\to\pi_{b}$ for all $b\in\intn{0}{c-1}$. Thus Proposition~\ref{ubiquitymultinom} implies the following result.

\begin{cor}\label{grintbarralseuret}
For every $t\in [1,\infty)$ and every probability vector $\pi=(\pi_{0},\ldots,\pi_{c-1})\in (0,1)^c$ ($c\geq 2$), the set $U_{\pi,t}$ defined by~(\ref{defupit}) contains a set of the class $\grint^{h_{\alpha/t,3f_{\eta}}}(\R)$ where $\alpha=-\sum_{b=0}^{c-1} \pi_{b}\log_{c}\pi_{b}$ and $f_{\eta}(r)=(-\log r)^{\eta-1/8}$ for each $\eta\in (0,1/8)$.
\end{cor}

Examining the proof of Proposition~\ref{ubiquitymultinom}, one easily checks that the statement of Corollary~\ref{grintbarralseuret} remains valid if the Besicovitch condition bears on the approximated reals rather than on the rational approximates, that is, if $U_{\pi,t}$ is replaced by
$$
U_{\pi,t}'=\left\{ x\in\R \:\Biggl|\: \begin{array}{l}
\exists (p_{n}/q_{n})_{n\in\N}\text{ irreducible} \:|\: q_{n}\to\infty \\
\forall n \qquad |x-p_{n}/q_{n}|\leq {q_{n}}^{-2t} \\
\forall b \qquad \sigma_{b,\lfloor 2\log_{c}q_{n} \rfloor}(x) \to \pi_{b}
\end{array}\right\}.
$$
In view of the fact that $\grint^{h_{\alpha/t,3f_{\eta}}}(\R)\subset\grintfalc^{\alpha/t}(\R)$ for every $\eta\in (0,1/8)$, we thus obtain a refinement of Theorem~1.2 in~\cite{Barral:2006fk}.

Let us now investigate the large intersection properties of the set of all reals that are $2t$-approximable by irreducible rationals $p_{n}/q_{n}$, $n\in\N$, such that the frequencies
$$
\bar\sigma_{b}((p_{n},q_{n})_{n\in\N})=\lim_{n\to\infty}\sigma_{b,\lfloor 2\log_{c}q_{n} \rfloor}\left(\frac{p_{n}}{q_{n}}\right),
$$
for $b\in\intn{0}{c-1}$, satisfy some relations similar to those introduced in~\cite{Barreira:2002fk}. To fix ideas, let $c=3$ and let $V_{t}$ be the set of all reals $x$ such that $|x-p_{n}/q_{n}|\leq {q_{n}}^{-2t}$ for some sequence $(p_{n}/q_{n})_{n\in\N}$ of irreducible fractions with
$$
\bar\sigma_{0}((p_{n},q_{n})_{n\in\N})>0 \qquad\text{and}\qquad \bar\sigma_{1}((p_{n},q_{n})_{n\in\N})=2\,\bar\sigma_{2}((p_{n},q_{n})_{n\in\N})>0.
$$
It is straightforward to establish that $V_{t}\supseteq U_{(1-3\sigma,2\sigma,\sigma),t}$ for all $\sigma\in (0,1/3)$. Let $\alpha(\sigma)=-(1-3\sigma)\log_{3}(1-3\sigma)-2\sigma\log_{3}(2\sigma)-\sigma\log_{3}\sigma$. Corollary~\ref{grintbarralseuret} ensures that $V_{t}$ contains a set of the class $\grint^{h_{\alpha(\sigma)/t,3f_{\eta}}}(\R)$ for all $\eta\in (0,1/8)$. In particular, it contains a set of $\grint^{h_{\alpha(\sigma_{\star})/t,3f_{\eta}}}(\R)$ where $\sigma_{\star}$ is chosen to maximize $\alpha$ (to be specific, $\sigma_{\star}=(2^{4/3}+9-3\cdot 2^{2/3})/31$).

\subsection{Gibbs measures}

We begin by giving a brief account of Gibbs measures. Given an integer $c\geq 2$, let $\sigma(x)=cx\modu\Z^d$ for each $x\in [0,1)^d$. Let $f$ be a $\Z^d$-periodic H\"older continuous function. The Birkhoff sums associated with $f$ are defined by $S_{j}^f=\sum_{k=0}^{j-1} f\circ \sigma^k$ for all $j\in\N$. In addition, let $D_{j}^f=\exp\circ S_{j}^f$ and
$$
\mu^f_{j}({\rm d}x)=\frac{\ind_{[0,1)^d}(x) D_{j}^f(x) {\rm d}x}{\int_{[0,1)^d} D_{j}^f(y) {\rm d}y}.
$$
The iterates of the Ruelle operator can be computed explicitly in terms of $(\mu^f_{j})_{j\in\N}$ and the Ruelle-Perron-Frobenius theorem shows that this sequence converges weakly to a Borel probability measure $\mu_{f}$ with support $[0,1]^d$, see~\cite{Parry:1990fk}. Moreover, $\mu_{f}$ is ergodic and is a Gibbs measure, i.e. there is a positive real $C$ such that
\begin{equation}\label{gibbsencadr}
\frac{1}{C}\leq \frac{\mu_{f}(\lambda^c_{j}(x))}{e^{-j P_{f}(1)} D_{j}^f(x)} \leq C
\end{equation}
for all $j\in\N$ and all $x\in [0,1)^d$, where $P_{f}$ denotes the pressure function associated with $f$, that is,
$$
P_{f}: q\mapsto d\log c+\lim_{j\to\infty}\frac{1}{j}\log\int_{[0,1)^d} D_{j}^{qf}(x) {\rm d}x.
$$
The multifractal analysis of $\mu_{f}$ is given in~\cite{Fan:1997uq,Fan:1998kx} and involves a function $\tau_{\mu_{f}}$ which can be expressed in terms of $P_{f}$ thanks to~(\ref{gibbsencadr}). Specifically,
$$
\forall q\in\R \qquad \tau_{\mu_{f}}(q)=-\lim_{j\to\infty} \frac{1}{j}\log_{c}\sum_{\lambda\subseteq [0,1) \atop \gene{\lambda}_{c}=j} \left( \mu_{f}(\lambda) \right)^q=\frac{1}{\log c}\left( q P_{f}(1) - P_{f}(q) \right)
$$
where the sum is over all $c$-adic subcubes of $[0,1)^d$ with generation $j$. Given $q\in\R$, let $\alpha=\tau_{\mu_{f}}'(q)$ and $\beta=q\tau_{\mu_{f}}'(q)-\tau_{\mu_{f}}(q)$. Then $\beta$ is the Hausdorff dimension of the set $V_{\alpha}$ of all points $x\in [0,1)^d$ such that $-\frac{1}{j}\log_{c}\mu_{f}(\lambda^c_{j}(x)) \to \alpha$.

Let us investigate the large intersection properties of the corresponding sets $\tilde E_{t}$ and $E_{t}$ which are defined in Section~\ref{grintheterubiq}. Let $I$ be a denumerable set and let $(x_{i},r_{i})_{i\in I}\in\S^0_{d}(I)$. Assume that for each $x\in [0,1]^d$ there are infinitely many $i\in I$ such that $\|x-x_{i}\|\leq r_{i}/2$. As a result,~(\ref{condcov}) holds for all $m\in\mesb$. Let us show that~(\ref{condm})-(\ref{condselfsim}) hold for $m=\mu_{qf}$ and $\mu=\mu_{f}$. The law of the iterated logarithm implies that for some $\kappa'>0$ and $m$-almost every $x\in [0,1]^d$, there is an integer $j(x)$ such that, for every integer $j\geq j(x)$,
$$
\left| \frac{1}{j} S^{qf}_{j}(x) - q P'_{f}(q) \right| \leq \kappa'\varsigma(c^{-j})
$$
where $\varsigma$ is given by~(\ref{correcvarsigma}), see~\cite{Denker:1984fk}. Then, as $f$ is H\"older continuous, for some $\kappa\geq\kappa'$ and $m$-almost every $x\in [0,1]^d$, there is an integer $j(x)$ such that
$$
\sup_{z\in 3\lambda^c_{j}(x)} \left| \frac{1}{j} S^{qf}_{j}(z) - q P'_{f}(q) \right| \leq \kappa\varsigma(c^{-j})
$$
for every integer $j\geq j(x)$. Using~(\ref{gibbsencadr}) with the measure $m$, we obtain~(\ref{condm}) for $\ph=2\kappa\varsigma/\log c$. Likewise, making use of~(\ref{gibbsencadr}) with the measure $\mu$, we get~(\ref{condmu}) for $\psi=2\kappa\varsigma/(|q|\log c)$ provided that $q\neq 0$. Furthermore, some routine calculations show that~(\ref{condselfsim}) holds. Hence, $(x_{i},r_{i})_{i\in I}$ is a heterogeneous ubiquitous system with respect to $(\mu,\alpha,\beta,\ph)$ if $q\neq 0$. Theorem~\ref{UBIQUITYHETER} thus ensures that $\tilde E_{t}\in\grint^{h_{\frac{\beta}{t},\frac{6\kappa}{\log c}\varsigma}}((0,1)^d)$ and $E_{t}\in\grint^{h_{\frac{\beta}{t},\frac{6\kappa}{\log c}\varsigma}}(\R^d)$ for all $t\in [1,\infty)$ and some $M\in [1,\infty)$.

The previous results can be expressed in terms of the Birkhoff sums $S^f_{j}$. Specifically, let $I^\birkhoff_{f,q}$ denote the set of all $i\in I$ such that
$$
\left| \frac{1}{\lfloor -\log_{c}r_{i} \rfloor} S^{f}_{\lfloor -\log_{c}r_{i} \rfloor}(x_{i}) - P'_{f}(q) \right| \leq 2\frac{\kappa}{|q|}\varsigma(r_{i}).
$$
Note that $I^\birkhoff_{f,0}=I$. The following proposition is established in Section~\ref{grintheterproofubiq}.

\begin{prp}\label{ubiquitygibbs}
Let $f$ be a $\Z^d$-periodic H\"older continuous function, let $I$ be a denumerable set and let $(x_{i},r_{i})_{i\in I}\in\S^0_{d}(I)$. Assume that each $x\in [0,1]^d$ enjoys $\|x-x_{i}\|\leq r_{i}/2$ for infinitely many $i\in I$. Then, for all $q\in\R$ and all $t\in [1,\infty)$,
$$
\tilde E^\birkhoff_{t}=\left\{ x\in [0,1]^d \:\bigl|\: \|x-x_{i}\|<{r_{i}}^t\text{ for infinitely many }i\in I^\birkhoff_{f,q} \right\}
$$
belongs to $\grint^{h_{\frac{\beta}{t},\frac{6\kappa}{\log c}\varsigma}}((0,1)^d)$ and $E^\birkhoff_{t}=\Z^d+\tilde E^\birkhoff_{t}$ belongs to $\grint^{h_{\frac{\beta}{t},\frac{6\kappa}{\log c}\varsigma}}(\R^d)$, where $\beta=q\tau_{\mu_{f}}'(q)-\tau_{\mu_{f}}(q)$ and $\varsigma$ is given by~(\ref{correcvarsigma}).
\end{prp}

A typical application of Proposition~\ref{ubiquitygibbs} is the fact that the set of all real numbers that are $2t$-approximable by rationals $p_{n}/q_{n}$, $n\in\N$, such that
$$
\lim_{n\to\infty} \frac{1}{\lfloor 2\log_{c}q_{n} \rfloor} S^{f}_{\lfloor 2\log_{c}q_{n} \rfloor}\left(\frac{p_{n}}{q_{n}}\right) = P'_{f}(q)
$$
contains a set of the class $\grint^{h_{\frac{\beta}{t},\frac{6\kappa}{\log c}\varsigma}}(\R)$.

\section{Proof of Proposition~\ref{HAUSIMDCGRINT}}\label{grintheterproofsets}

We use some ideas of the proof of the implication~(b)$\Rightarrow$(c) of Theorem~B in~\cite{Falconer:1994hx}, but the situation is a little more complex here because the gauge functions we consider are more general than those of~\cite{Falconer:1994hx}. Let $\beta\in (0,d]$ and $\ph\in\Phi$. It is straightforward to check that for some $\eta\in (0,\eps_{h_{\beta,\ph}}]$,
\begin{equation}\label{condlimsupbetar1r2}
\forall r_{1},r_{2}\in (0,\eta) \qquad r_{1}\leq r_{2} \quad\Longrightarrow\quad \limsup_{\kappa\to 0}\frac{h_{\beta,\ph}(\kappa r_{1})}{h_{\beta,\ph}(\kappa r_{2})}\leq \frac{h_{\beta,\ph}(r_{1})}{h_{\beta,\ph}(r_{2})}.
\end{equation}
We now let $F$ be a $G_{\delta}$-set such that~(\ref{hauhcapfnfpos}) holds for every sequence $(f_{n})_{n\in\N}$ of similarities in $D_{c}$ and show that $F\in\grint^{h_{\beta,\ph}}(\R^d)$. By Lemma~10 in~\cite{Durand:2007uq}, it suffices to prove that $\netm^{h_{\beta,\ph}}_{\infty}(F\cap\lambda)=\netm^{h_{\beta,\ph}}_{\infty}(\lambda)$ for each $c$-adic cube $\lambda$ of diameter less than $\eta$. Let us assume that there is a $c$-adic cube $\lambda$ of diameter less than $\eta$ and a real number $\alpha\in (0,1)$ with $\netm^{h_{\beta,\ph}}_{\infty}(F\cap\lambda) < \alpha \netm^{h_{\beta,\ph}}_{\infty}(\lambda)$. For $G\subseteq\R^d$, let $R^\lambda_{c}(G)$ be the set of all sequences $(\lambda_{p})_{p\in\N}$ in $\Lambda_{c}\cup\{\emptyset\}$ such that $G\cap\lambda\subseteq\bigsqcup_{p}\lambda_{p}\subseteq\lambda$ (i.e. the sets $\lambda_{p}$, $p\in\N$, are disjoint, included in $\lambda$ and cover $G$). Thanks to Lemma~8 and Lemma~9 in~\cite{Durand:2007uq}, for some $(\lambda_{p})_{p\in\N}\in R^\lambda_{c}(F)$, we have
\begin{equation}\label{majsumglambdai}
\sum_{p=1}^\infty h_{\beta,\ph}(\diam{\lambda_{p}}) \leq \alpha h_{\beta,\ph}(\diam{\lambda})
\end{equation}
and $\lambda_{p}\subseteq\lambda$ for all $p$. Let $P$ be the set of all $p\in\N$ such that $\lambda_{p}\neq\emptyset$ and, for each $p\in P$, let $f_{p}$ be the central dilation that maps $\lambda$ to $\lambda_{p}$. Furthermore, let $f_{p_{1},\ldots,p_{s}}=f_{p_{1}}\circ\ldots\circ f_{p_{s}}$ for $s\in\N$ and $(p_{1},\ldots,p_{s})\in P^s$. This is a central dilation which maps $\lambda$ to one of its proper $c$-adic subcubes. In addition, let $f_{p_{1},\ldots,p_{s}}$ denote the identity function if $s=0$.

Let $\gamma$ be a positive real number such that $(1+\gamma)\alpha<1$. Owing to~(\ref{condlimsupbetar1r2}), for some $\bar\kappa\in (0,\eps_{h_{\beta,\ph}})$ and every $\kappa\in (0,\bar\kappa)$, we have
\begin{equation}\label{condlimsupbetar1r2bis}
\forall r_{1},r_{2}\in (0,\eta) \qquad r_{1}\leq r_{2} \quad\Longrightarrow\quad \frac{h_{\beta,\ph}(\kappa r_{1})}{h_{\beta,\ph}(\kappa r_{2})}\leq (1+\gamma)\frac{h_{\beta,\ph}(r_{1})}{h_{\beta,\ph}(r_{2})}.
\end{equation}
By Theorem~4 in~\cite{Rogers:1970wb}, we build a finite outer measure on $\lambda$ by setting
$$
\forall G\subseteq\lambda \qquad m_{\kappa}(G) = \inf_{(\nu_{p})_{p\in\N}} \sum_{p=1}^\infty h_{\beta,\ph}(\diam{\nu_{p}})
$$
where the infimum is taken over all sequences $(\nu_{p})_{p\in\N}$ in $\Lambda_{c}\cup\{\emptyset\}$ such that $G\subseteq\bigcup_{p}\nu_{p}\subseteq\lambda$ and $\diam{\nu_{p}}<\kappa\diam{\lambda}$ for all $p$. Let $(\nu_{p})_{p\in\N}$ be such a sequence for $G=f_{p_{2},\ldots,p_{s}}(F\cap\lambda)$. Then the sets $f_{p_{1}}(\nu_{p})\subseteq\lambda$, $p\in\N$, belong to $\Lambda_{c}\cup\{\emptyset\}$, are of diameter $\diam{\nu_{p}}\cdot\diam{\lambda_{p_{1}}}/\diam{\lambda}<\kappa\diam{\lambda}$ and cover $f_{p_{1},\ldots,p_{s}}(F\cap\lambda)$. Thus
$$
m_{\kappa}\left( f_{p_{1},\ldots,p_{s}}(F\cap\lambda) \right) \leq \sum_{p=1}^\infty h_{\beta,\ph}\left( \frac{\diam{\nu_{p}}}{\diam{\lambda}}\diam{\lambda_{p_{1}}} \right) \leq (1+\gamma)\sum_{p=1}^\infty \frac{h_{\beta,\ph}(\diam{\lambda_{p_{1}}})}{h_{\beta,\ph}(\diam{\lambda})} h_{\beta,\ph}(\diam{\nu_{p}})
$$
owing to~(\ref{condlimsupbetar1r2bis}) together with $\diam{\nu_{p}}/\diam{\lambda}<\kappa<\bar\kappa$ and $\diam{\lambda_{p_{1}}}<\diam{\lambda}<\eta$. Taking the infimum over all sequences $(\nu_{p})_{p\in\N}$ in the right-hand side, we get
$$
m_{\kappa}(f_{p_{1},\ldots,p_{s}}(F\cap\lambda))\leq (1+\gamma)\frac{h_{\beta,\ph}(\diam{\lambda_{p_{1}}})}{h_{\beta,\ph}(\diam{\lambda})} m_{\kappa}(f_{p_{2},\ldots,p_{s}}(F\cap\lambda)).
$$
The procedure is iterated so as to obtain
$$
m_{\kappa}\left( f_{p_{1},\ldots,p_{s}}(F\cap\lambda) \right) \leq m_{\kappa}(F\cap\lambda)\left(\frac{1+\gamma}{h_{\beta,\ph}(\diam{\lambda})}\right)^s\prod_{s'=1}^s h_{\beta,\ph}(\diam{\lambda_{p_{s'}}}).
$$
Hence, summing over all $s$-tuples $(p_{1},\ldots,p_{s})\in P^s$ and using~(\ref{majsumglambdai}), we have
\begin{equation*}\begin{split}
\sum_{(p_{1},\ldots,p_{s})\in P^s} m_{\kappa}\left( f_{p_{1},\ldots,p_{s}}(F\cap\lambda) \right) &\leq \left(\frac{1+\gamma}{h_{\beta,\ph}(\diam{\lambda})}\sum_{p=1}^\infty h_{\beta,\ph}(\diam{\lambda_{p}})\right)^s m_{\kappa}(F\cap\lambda) \\
&\leq \left((1+\gamma)\alpha\right)^s m_{\kappa}(\lambda).
\end{split}\end{equation*}
In addition, let $x\in\lambda\cap f_{p_{1},\ldots,p_{q}}(F)$ for every integer $q\geq 0$ and every $q$-tuple $(p_{1},\ldots,p_{q})\in P^q$. It is straightforward to show by induction on $s\geq 0$ that $x\in f_{p_{1},\ldots,p_{s}}(F\cap\lambda)$ for some $s$-tuple $(p_{1},\ldots,p_{s})\in P^s$. As a result,
$$
\lambda\cap\bigcap_{q=0}^\infty\bigcap_{(p_{1},\ldots,p_{q})\in P^q} f_{p_{1},\ldots,p_{q}}(F) \subseteq \bigcup_{(p_{1},\ldots,p_{s})\in P^s} f_{p_{1},\ldots,p_{s}}(F\cap\lambda)
$$
so that
$$
m_{\kappa}\left( \lambda\cap\bigcap_{q=0}^\infty\bigcap_{(p_{1},\ldots,p_{q})\in P^q} f_{p_{1},\ldots,p_{q}}(F) \right) \leq \left((1+\gamma)\alpha\right)^s m_{\kappa}(\lambda).
$$
Let $\hau^{h_{\beta,\ph}}_{\kappa\diam{\lambda}}$ denote the Hausdorff pre-measure associated with $h_{\beta,\ph}$ and defined in terms of coverings by sets of diameter less than $\kappa\diam{\lambda}$. Note that $m_{\kappa}(G)\geq \hau^{h_{\beta,\ph}}_{\kappa\diam{\lambda}}(G)$ for every $G\subseteq\lambda$. Thus, letting $s\to\infty$ and $\kappa\to 0$, we obtain
$$
\hau^{h_{\beta,\ph}}\left( \lambda\cap\bigcap_{q=0}^\infty\bigcap_{(p_{1},\ldots,p_{q})\in P^q} f_{p_{1},\ldots,p_{q}}(F) \right) = 0.
$$
Let $(\nu_{n})_{n\in\N}$ be an enumeration of all $c$-adic cubes of generation $\gene{\lambda}_{c}$ and, for every $n\in\N$, let $t_{n}$ denote the translation that maps $\lambda$ to $\nu_{n}$. For each $n_{0}\in\N$, we have
$$
\nu_{n_{0}}\cap\bigcap_{n=1}^\infty\bigcap_{q=0}^\infty\bigcap_{(p_{1},\ldots,p_{q})\in P^q} t_{n}\circ f_{p_{1},\ldots,p_{q}}(F) \subseteq t_{n_{0}}\left(\lambda\cap\bigcap_{q=0}^\infty\bigcap_{(p_{1},\ldots,p_{q})\in P^q} f_{p_{1},\ldots,p_{q}}(F) \right)
$$
so that the left-hand side has zero Hausdorff $h_{\beta,\ph}$-measure. Since the cubes $\nu_{n_{0}}$, $n_{0}\in\N$, form a partition of $\R^d$, we end up with
$$
\hau^{h_{\beta,\ph}}\left( \bigcap_{n=1}^\infty\bigcap_{q=0}^\infty\bigcap_{(p_{1},\ldots,p_{q})\in P^q} t_{n}\circ f_{p_{1},\ldots,p_{q}}(F) \right) = 0
$$
which contradicts~(\ref{hauhcapfnfpos}) because $t_{n}\circ f_{p_{1},\ldots,p_{q}}\in D_{c}$ for every $(p_{1},\ldots,p_{q})\in P^q$, every integer $q\geq 0$ and every $n\in\N$. Proposition~\ref{HAUSIMDCGRINT} is proven.

\section{Proof of Theorem~\ref{UBIQUITYHETER}}\label{grintheterproofubiq}

In this section, $\R^d$ is endowed with the supremum norm. Let $I$ be a denumerable set and let $(x_{i},r_{i})_{i\in I}\in\S^0_{d}(I)$ be a heterogeneous ubiquitous system with respect to some $(\mu,\alpha,\beta,\ph)\in\mesb\times (0,\infty)\times (0,d]\times\Phi$. Thus there are an integer $c\geq 2$, a function $\psi\in\Psi$ and a measure $m\in\mesb$ such that~(\ref{condcov})-(\ref{condselfsim}) hold, see Section~\ref{grintheterubiq}.

\subsection{Preliminaries}\label{ubiquitylemmas}

We begin by introducing some extra notations and establishing a series of technical lemmas which are called upon at various points in the proof of Theorem~\ref{UBIQUITYHETER}. For every Borel subset $B$ of $\R^d$, let
$$
\bar m(B) = \sum_{p\in\Z^d} m\left( (p+B) \cap [0,1)^d \right).
$$
Then $\bar m$ is a $\sigma$-finite Borel measure on $\R^d$. Owing to~(\ref{condcov}), the Borel set
\begin{equation*}\begin{split}
X &= \Z^d+\left\{x\in [0,1]^d \:\bigl|\: \|x-x_{i}\|\leq r_{i}/2\text{ for infinitely many }i\in I \right\} \\
&= \left\{ x\in\R^d \:\bigl|\: \|x-p-x_{i}\|\leq r_{i}/2\text{ for infinitely many }(i,p)\in I\times\Z^d \right\}
\end{split}\end{equation*}
has full $\bar m$-measure in $\R^d$. Moreover,~(\ref{condm}) implies that for some Borel set $S\subseteq [0,1)^d$ of full $m$-measure in $[0,1)^d$ and every $x\in S$, there is an integer $j(x)\in\N$ such that for every integer $j\geq j(x)$ and every $k\in\intn{0}{c^j-1}^d$,
$$
\lambda^c_{j,k} \subseteq 3\lambda^c_{j}(x) \quad\Longrightarrow\quad m(\lambda^c_{j,k})\leq h_{\beta,\ph}(\diam{\lambda^c_{j,k}}).
$$
Likewise,~(\ref{condmu}) shows that for some Borel set $\Sigma$ of full $\bar m$-measure in $\R^d$ and every $x\in\Sigma$, there exists an integer $j(x)\in\N$ such that for every integer $j\geq j(x)$ and every $k\in\intn{0}{c^j-1}^d$,
$$
\lambda^c_{j,k} \subseteq 3\lambda^c_{j}(x\modu\Z^d) \quad\Longrightarrow\quad \diam{\lambda^c_{j,k}}^{\alpha+\psi(\diam{\lambda^c_{j,k}})} \leq \mu(\lambda^c_{j,k}) \leq \diam{\lambda^c_{j,k}}^{\alpha-\psi(\diam{\lambda^c_{j,k}})}
$$
where $x\modu\Z^d$ denotes the unique point in $[0,1)^d$ whose coordinates are congruent to those of $x$ modulo $1$.

\begin{lem}\label{fXSfSigmamespleine}
For every $f\in D_{c}$, the set $f(X)\cap S\cap f(\Sigma)$ has full $m$-measure in $[0,1)^d$.
\end{lem}

\begin{proof}
First recall that $m([0,1)^d\backslash S)=0$. Moreover, there are $\lambda,\lambda'\in\Lambda_{c}$ with $0<\gene{\lambda}_{c}\leq\gene{\lambda'}_{c}$ such that the dilation $f$ maps $\lambda$ to $\lambda'$. In particular, the generation of $\lambda'$ is positive, so that
$$
m\left( [0,1)^d \backslash f(X) \right) = \sum_{\nu\subseteq [0,1)^d \atop \gene{\nu}_{c}=\gene{\lambda'}_{c}} m\left( \nu \backslash f(X) \right).
$$
Let $\nu$ be a $c$-adic subcube of $[0,1)^d$ enjoying $\gene{\nu}_{c}=\gene{\lambda'}_{c}$. Then $f^{-1}(\nu)$ is a $c$-adic cube with generation equal to that of $\lambda$ and so is the Borel subset $t_{\nu}(f^{-1}(\nu))$ of $[0,1)^d$, where $t_{\nu}$ is the translation which maps $\omega_{f^{-1}(\nu)}(f^{-1}(\nu))$ to $\omega_{\nu}(\nu)=[0,1)^d$. Furthermore, $f={\omega_{\nu}}^{-1}\circ \omega_{t_{\nu}(f^{-1}(\nu))}\circ t_{\nu}$ so that
\begin{equation*}\begin{split}
m\left( \nu \backslash f(X) \right) &= m\restr{\nu}\left( \nu \backslash {\omega_{\nu}}^{-1}\circ \omega_{t_{\nu}(f^{-1}(\nu))}\circ t_{\nu}(X) \right) \\
&= m\restr{\nu}\circ{\omega_{\nu}}^{-1}\left( [0,1)^d \backslash \omega_{t_{\nu}(f^{-1}(\nu))}\circ t_{\nu}(X) \right).
\end{split}\end{equation*}
As a result, $m(\nu\backslash f(X))=0$ if and only if $m([0,1)^d \backslash \omega_{t_{\nu}(f^{-1}(\nu))}\circ t_{\nu}(X))=0$, on account of~(\ref{condselfsim}). Meanwhile,
$$
m\left( [0,1)^d \backslash \omega_{t_{\nu}(f^{-1}(\nu))}\circ t_{\nu}(X) \right) = m^{t_{\nu}(f^{-1}(\nu))}\left( t_{\nu}(f^{-1}(\nu) \backslash X) \right).
$$
Hence $m([0,1)^d \backslash \omega_{t_{\nu}(f^{-1}(\nu))}\circ t_{\nu}(X))=0$ if and only if $m(t_{\nu}(f^{-1}(\nu) \backslash X))=0$, owing to~(\ref{condselfsim}) again. In addition, observe that
$$
m\left( t_{\nu}(f^{-1}(\nu) \backslash X) \right) = \bar m\left( f^{-1}(\nu) \backslash X \right) \leq \bar m\left(\R^d \backslash X \right)=0
$$
so that $m(\nu\backslash f(X))=0$ for every $c$-adic subcube $\nu$ of $[0,1)^d$ with generation equal to that of $\lambda'$. It follows that $m( [0,1)^d \backslash f(X) )=0$. Likewise, $m( [0,1)^d \backslash f(\Sigma) )=0$.
\end{proof}

Let $M\in (0,\infty)$ and $t\in [1,\infty)$. The following lemma shows that the set $E_{t}$ given by~(\ref{etlimsupballsim}) satisfies the same large intersection properties as
\begin{equation}\label{etlimsupballsimweak}\begin{split}
E_{t}^{\weak} &= \Z^d+\left\{ x\in [0,1]^d \:\bigl|\: \|x-x_{i}\|<{r_{i}}^t\text{ for infinitely many }i\in I^{\mu,\alpha,\weak}_{M,\psi} \right\} \\
&= \left\{ x\in\R^d \:\bigl|\: \|x-p-x_{i}\|<{r_{i}}^t\text{ for infinitely many }(i,p)\in I^{\mu,\alpha,\weak}_{M,\psi}\times\Z^d \right\}
\end{split}\end{equation}
where $I^{\mu,\alpha,\weak}_{M,\psi}$ denotes the set of all $i\in I$ such that
\begin{equation}\label{ballsimweakcond}\begin{split}
M^{-1} (2r_{i})^{\alpha+\psi(2r_{i})} &\leq \mu\restr{[0,1)^d}\left( B(\ell+x_{i},r_{i}) \right) \\
&\leq \mu\restr{[0,1)^d}\left( \bar B(\ell+x_{i},r_{i}) \right) \leq M (2r_{i})^{\alpha-\psi(2r_{i})}
\end{split}\end{equation}
for some $\ell\in\{-1,0,1\}^d$, so it suffices to prove Theorem~\ref{UBIQUITYHETER} with $E_{t}^{\weak}$ instead of $E_{t}$.

\begin{lem}\label{etweakgrint}
Let $M\in (0,\infty)$, let $t\in [1,\infty)$ and let $h$ be a gauge function such that $E_{t}^{\weak}$ belongs to $\grint^h(\R^d)$. Then $E_{t}$ belongs to $\grint^h(\R^d)$ as well.
\end{lem}

\begin{proof}
First note that $E_{t}^{\weak}$ and $E_{t}$ are $G_{\delta}$-sets. Moreover, the open set
\begin{equation}\label{defFunioncubouv}
F = \Z^d+(0,1)^d = \bigsqcup_{q\in\Z^d} \left( q+(0,1)^d \right)
\end{equation}
has full Lebesgue measure in $\R^d$. Thus Proposition~11 and Theorem~1 in~\cite{Durand:2007uq} ensure that $E_{t}^{\weak}\cap F\in\grint^h(\R^d)$. Thanks to Proposition~1 in~\cite{Durand:2007uq}, it suffices to show that $E_{t}\supseteq E_{t}^{\weak}\cap F$. Let $x\in E_{t}^{\weak}\cap F$. There is an injective sequence $(i_{n},p_{n})_{n\in\N}$ in $I^{\mu,\alpha,\weak}_{M,\psi}\times\Z^d$ such that $\|x-p_{n}-x_{i_{n}}\|<{r_{i_{n}}}^t$ for all $n\in\N$. Let us write $x=q+\dot x$ where $q\in\Z^d$ and $\dot x\in (0,1)^d$. We have $B(\dot x,2\delta)\subseteq (0,1)^d$ for some $\delta\in (0,\infty)$ and ${r_{i_{n}}}^t<\delta$ for $n$ large enough. For all such $n$,
$$
p_{n}+x_{i_{n}}\in B(x,{r_{i_{n}}}^t) \subseteq q+B(\dot x,\delta)\subseteq q+(0,1)^d.
$$
As a consequence, $p_{n}=q$ and $x_{i_{n}}\in B(\dot x,\delta)$. Thus $\bar B(x_{i_{n}},r_{i_{n}})\subseteq B(\dot x,2\delta)\subseteq (0,1)^d$. Furthermore,~(\ref{ballsimweakcond}) holds for some $\ell\in\{-1,0,1\}^d$, since $i_{n}\in I^{\mu,\alpha,\weak}_{M,\psi}$. We necessarily have $\mu(B(\ell+x_{i_{n}},r_{i_{n}}))>0$ so that $\ell=(0,\ldots,0)$ and~(\ref{ballsimcond}) holds. It follows that $i_{n}\in I^{\mu,\alpha}_{M,\psi}$. As a result, $x\in E_{t}$.
\end{proof}

In the sequel, $M=\max( 3^d c^\alpha , 3^d c^d , 2^{1+3\alpha}c^\alpha )$ and $|f|$ denotes the ratio of any dilation $f\in D_{c}$. Note that $|f|\leq 1$.

\begin{lem}\label{lemfXSfSigma}
Let $f\in D_{c}$ and $z\in (0,1)^d\cap f(X)\cap S\cap f(\Sigma)$. Then for some infinite subset $I_{f}(z)$ of $I^{\mu,\alpha,\weak}_{M,\psi}$ and every $i\in I_{f}(z)$ the following properties hold:
\begin{itemize}
\item $\|f^{-1}(z)-p-x_{i}\|\leq r_{i}/2$ for some $p\in\Z^d$,
\item $\bar B(z,2|f|r_{i})\subseteq (0,1)^d$ and $m( \bar B(z,2|f|r_{i}) )\leq M h_{\beta,\ph}( 4|f|r_{i} )$.
\end{itemize}
\end{lem}

\begin{proof}
Let $x=f^{-1}(z)$ and $\dot x=x\modu\Z^d$. As $h_{\beta,\ph}\in\jauge_{d}$ and $\psi\in\Psi$ and by definition of $S$ and $\Sigma$, there is an integer $j_{0}\in\N$ such that the function $r\mapsto h_{\beta,\ph}(r)/r^d$ is nonincreasing on $(0,c^{-j_{0}})$, the function $r\mapsto r^{-\psi(r)}$ is nonincreasing on $(0,8c^{-j_{0}+1}]$, the inequality $(8cr)^{\psi(8cr)}\leq 2r^{\psi(r)}$ holds for all $r\in (0,c^{-j_{0}}]$ and for every integer $j\geq j_{0}$ and every $k\in\intn{0}{c^j-1}^d$,
$$
\left\{\begin{array}{lcl}
\lambda^c_{j,k} \subseteq 3\lambda^c_{j}(z) & \Longrightarrow & m(\lambda^c_{j,k})\leq h_{\beta,\ph}(\diam{\lambda^c_{j,k}}) \\[2mm]
\lambda^c_{j,k} \subseteq 3\lambda^c_{j}(\dot x) & \Longrightarrow & \diam{\lambda^c_{j,k}}^{\alpha+\psi(\diam{\lambda^c_{j,k}})} \leq \mu(\lambda^c_{j,k}) \leq \diam{\lambda^c_{j,k}}^{\alpha-\psi(\diam{\lambda^c_{j,k}})}
\end{array}\right. .
$$
Furthermore, by definition of $X$, for some infinite subset $I'$ of $I$ and every $i\in I'$, we have $x\in \bar B(p+x_{i},r_{i}/2)$ for some $p\in\Z^d$. In addition, since $z\in (0,1)^d$, we have $4r_{i}\leq c^{-j_{0}}$ and $\bar B(z,2|f|r_{i})\subseteq (0,1)^d$ for each $i$ in some infinite subset $I_{f}(z)$ of $I'$.

Let $i\in I_{f}(z)$ and let $j_{1}$ be the largest integer such that $4|f|r_{i}\leq c^{-j_{1}}$. Note that $j_{1}\geq j_{0}$. Thus $m(\lambda^c_{j_{1},k})\leq h_{\beta,\ph}(\diam{\lambda^c_{j_{1},k}})$ for each $k\in\intn{0}{c^{j_{1}}-1}^d$ with $\lambda^c_{j_{1},k} \subseteq 3\lambda^c_{j_{1}}(z)$. This results in $m(\bar B(z,2|f|r_{i})) \leq 3^d h_{\beta,\ph}( c^{-j_{1}} ) \leq M h_{\beta,\ph}(4|f|r_{i})$ because $\bar B(z,2|f|r_{i})\subseteq 3\lambda^c_{j_{1}}(z)$ and $r\mapsto r^{-d}h_{\beta,\ph}(r)$ is nonincreasing on $(0,c^{-j_{0}})$.

It remains to show that $I_{f}(z)\subseteq I^{\mu,\alpha,\weak}_{M,\psi}$. Let $i\in I_{f}(z)$, let $j_{2}$ be the smallest integer such that $c^{-j_{2}}\leq r_{i}/4$ and $j_{3}$ the largest integer such that $2r_{i}\leq c^{-j_{3}}$. We have $j_{2}\geq j_{3}\geq j_{0}$, so that $\mu(\lambda^c_{j_{2}}(\dot x))\geq (c^{-j_{2}})^{\alpha+\psi(c^{-j_{2}})}$ and $\mu(\lambda^c_{j_{3},k}) \leq (c^{-j_{3}})^{\alpha-\psi(c^{-j_{3}})}$ for each $k\in\intn{0}{c^{j_{3}}-1}^d$ with $\lambda^c_{j_{3},k} \subseteq 3\lambda^c_{j_{3}}(\dot x)$. Since $i\in I_{f}(z)\subseteq I'$, we have $x\in \bar B(p+x_{i},r_{i}/2)$ for some $p\in\Z^d$. Thus $\dot x=x-p+\ell\in\bar B(\ell+x_{i},r_{i}/2)$ for some $\ell\in\{-1,0,1\}^d$. As a result, $\lambda^c_{j_{2}}(\dot x) \subseteq B(\ell+x_{i},r_{i}) \subseteq \bar B(\ell+x_{i},r_{i}) \subseteq 3\lambda^c_{j_{3}}(\dot x)$ so that
\begin{equation*}\begin{split}
(c^{-j_{2}})^{\alpha+\psi(c^{-j_{2}})} &\leq \mu\restr{[0,1)^d}\left( B(\ell+x_{i},r_{i}) \right) \\
&\leq \mu\restr{[0,1)^d}\left( \bar B(\ell+x_{i},r_{i}) \right) \leq 3^d (c^{-j_{3}})^{\alpha-\psi(c^{-j_{3}})}.
\end{split}\end{equation*}
Recall that the function $r\mapsto r^{-\psi(r)}$ is nonincreasing on $(0,8c^{-j_{0}+1}]$ and that the inequality $(8cr)^{\psi(8cr)}\leq 2 r^{\psi(r)}$ holds for all $r\in (0,c^{-j_{0}}]$. Therefore the right-hand side is at most $M (2r_{i})^{\alpha-\psi(2r_{i})}$ and the left-hand side is at least $(2r_{i})^{\alpha+\psi(2r_{i})}/M$. We end up with~(\ref{ballsimweakcond}), so that $i\in I^{\mu,\alpha,\weak}_{M,\psi}$.
\end{proof}

\subsection{Proof of Theorem~\ref{UBIQUITYHETER}}

By Lemma~\ref{etweakgrint}, it suffices to establish that the set $E_{t}^{\weak}$ given by~(\ref{etlimsupballsimweak}) belongs to $\grint^{h_{\beta,\ph}}(\R^d)$ if $t=1$ and to $\grint^{h_{\beta/t,2\ph+\phi}}(\R^d)$ for all $\phi\in\Phi$ if $t>1$. Furthermore, Proposition~\ref{HAUSIMDCGRINT} ensures that it is enough to show that
\begin{equation}\label{hauxicapdcstricpos}
\hau^{h_{\beta,\ph}}\left( \bigcap_{n=1}^\infty f_{n}(E_{1}^{\weak}) \right)>0 \qquad\text{and}\qquad \hau^{h_{\beta/t,2\ph+\phi}}\left( \bigcap_{n=1}^\infty f_{n}(E_{t}^{\weak}) \right)>0
\end{equation}
for every sequence $(f_{n})_{n\in\N}$ of similarities in $D_{c}$.

Assume that $t=1$. By Lemma~\ref{fXSfSigmamespleine} and Lemma~\ref{lemfXSfSigma}, the set $(0,1)^d\cap f_{n}(X)\cap S\cap f_{n}(\Sigma)$ has full $m$-measure in $(0,1)^d$ and is included in $f_{n}(E_{1}^{\weak})$, for all $n\in\N$. Thus
$$
A=(0,1)^d\cap S \cap \bigcap_{n=1}^\infty \left( f_{n}(X)\cap f_{n}(\Sigma)\right)
$$
has full $m$-measure in $(0,1)^d$ and is included in $\bigcap_{n} f_{n}(E_{1}^{\weak})$. Furthermore, imitating the part of the proof of Lemma~\ref{lemfXSfSigma} which deals with $m$, one easily shows that
$$
\sup_{x\in A}\limsup_{r\to 0} \frac{m(B(x,r))}{h_{\beta,\ph}(2r)}<\infty.
$$
Then Lemma~\ref{massdistprinc} below implies that $\hau^{h_{\beta,\ph}}(A)>0$. Hence~(\ref{hauxicapdcstricpos}) holds for $t=1$.

\begin{lem}\label{massdistprinc}
Let $F$ be a Borel subset of $\R^d$ and let $\pi$ be a finite Borel measure on $\R^d$ with $\pi(F)>0$. Then $\hau^h(F)>0$ for any $h\in\jauge_{d}$ such that
$$
\sup_{x\in F}\limsup_{r\to 0} \frac{\pi(B(x,r))}{h(2r)}<\infty.
$$
\end{lem}

\begin{proof}
This is a straightforward generalization of Proposition~4.9 in~\cite{Falconer:2003oj}.
\end{proof}

From now on, we suppose that $t>1$. Let $\phi\in\Phi$. With a view to applying Lemma~\ref{massdistprinc} in order to establish~(\ref{hauxicapdcstricpos}), we shall build a generalized Cantor set $K$ along with a Borel probability measure $\pi$ supported on $K$ such that
\begin{equation}\label{eqpropKpi}
K\subseteq\bigcap_{n=1}^\infty f_{n}(E_{t}^{\weak}) \quad\text{and}\quad \exists C\in (0,\infty) \quad \forall B \quad \pi(B) \leq C h_{\beta/t,2\ph+\phi}(\diam{B})
\end{equation}
where $B$ denotes an open ball of small radius. To this end, we need to introduce some additional notations. Recall that the set $S$ has full $m$-measure in $(0,1)^d$. Meanwhile, $S=\bigcup_{j_{0}}\uparrow S^{\lambda}_{j_{0}}$ where $\lambda=[0,1)^d$ and
$$
S^{\lambda}_{j_{0}} = \left\{ x\in\interieur{\lambda} \:\Biggl|\:
\begin{array}{l}
\forall j\geq\gene{\lambda}_{c}+j_{0} \quad \forall k\in\intn{0}{c^j-1}^d \\
\lambda^c_{j,k} \subseteq 3\lambda^c_{j}(x)\cap\lambda \Longrightarrow m^{\lambda}(\lambda^c_{j,k})\leq h_{\beta,\ph}\left(\frac{\diam{\lambda^c_{j,k}}}{\diam{\lambda}}\right)
\end{array}
\right\}.
$$
Together with the fact that $\ph\in\Phi$, this implies that $r\mapsto r^{-\ph(r)}$ is nonincreasing on $(0,c^{-j_{0}+1}]$ and $m(S^{[0,1)^d}_{j_{0}})\geq |m|/2$ for some $j_{0}\in\N$, where $|m|=m((0,1)^d)$. Observe that $|m^\lambda|=m^\lambda(\interieur{\lambda})=|m|$ and $m^\lambda(S^{\lambda}_{j_{0}})\geq |m^\lambda|/2=|m|/2$ for every $c$-adic subcube $\lambda$ of $[0,1)^d$. Furthermore, since $\phi/(1+d)\in\Phi$, there is an integer $j_{1}\geq j_{0}$ such that the function $r\mapsto r^{-\phi(r)/(1+d)}$ is nonincreasing and greater than $c^{j_{0}}$ on $(0,c^{-j_{1}}]$ and such that the function $h_{\beta/t,2\ph+\phi/(1+d)}$ increases in the same interval. In addition, let $(n_{q})_{q\in\N}=(0,0,1,0,1,2,0,1,2,3,0,1,2,3,4,\ldots)$.

The construction of the generalized Cantor set $K$ calls upon the following covering theorem of Besicovitch, see~\cite[Theorem~2.7]{Mattila:1995fk}.

\begin{thm}[Besicovitch]\label{bescov}
Let $A$ be a bounded subset of $\R^d$ and let $\mathcal{B}$ be a family of closed balls such that each point in $A$ is the center of some ball of $\mathcal{B}$. Then, there are families $\mathcal{B}_{1},\ldots,\mathcal{B}_{Q(d)}\subset\mathcal{B}$ such that:\begin{itemize}
\item for every $\ell\in\intn{1}{Q(d)}$, the balls of $\mathcal{B}_{\ell}$ are disjoint;
\item the balls of $\mathcal{B}_{\ell}$, $\ell\in\intn{1}{Q(d)}$, cover $A$.
\end{itemize}
Here, $Q(d)$ denotes a positive integer which depends on $d$ only.
\end{thm}

We now detail the construction of $K$ and $\pi$. Let $G_{0}=\{[0,1)^d\}$ and $\pi([0,1]^d)=1$.

\subsubsection*{Step 1}

As $\phi/(1+d)\in\Phi$, for some integer $j\geq\max(j_{0}+5,j_{1},(\log_{c}2)/(t-1))$,
\begin{equation}\label{majcrocetap1}
\forall r\in (0,c^{-j}]\qquad  r^{-\phi(r)/(1+d)}\geq M \kappa^{\beta/t} \frac{4Q(d)}{|m|}
\end{equation}
where $\kappa=6 c^2 4^t$. Moreover, by Lemma~\ref{fXSfSigmamespleine}, the set $E^{[0,1)^d}=f_{n_{1}}(X)\cap S^{[0,1)^d}_{j_{0}}\cap f_{n_{1}}(\Sigma)$ has full $m$-measure in $S^{[0,1)^d}_{j_{0}}$ and, by Lemma~\ref{lemfXSfSigma}, for each point $z\in E^{[0,1)^d}$, there exists $i_{z}\in I_{f_{n_{1}}}(z)\subseteq I^{\mu,\alpha,\weak}_{M,\psi}$ such that $2r_{i_{z}}\leq c^{-j}$ and:\begin{itemize}
\item $\exists p_{z}\in\Z^d \quad \|{f_{n_{1}}}^{-1}(z)-p_{z}-x_{i_{z}}\|\leq r_{i_{z}}/2$,
\item $B_{z}=\bar B(z,2|f_{n_{1}}|r_{i_{z}})\subseteq (0,1)^d \quad\text{and}\quad m(B_{z})\leq M h_{\beta,\ph}( \diam{B_{z}} )$.
\end{itemize}
In addition, let $F_{z}$ be the open ball with center $f_{n_{1}}(p_{z}+x_{i_{z}})$ and radius $(|f_{n_{1}}|r_{i_{z}})^t$. Observe that $F_{z}\subseteq B_{z}\cap f_{n_{1}}(B(p_{z}+x_{i_{z}},{r_{i_{z}}}^t))$. Applying Besicovitch's covering theorem, we obtain families $\mathcal{B}_{1},\ldots,\mathcal{B}_{Q(d)}$ of points in $E^{[0,1)^d}$ such that the balls $B_{z}$, for $z\in\mathcal{B}_{\ell}$ and $\ell\in\intn{1}{Q(d)}$, cover $E^{[0,1)^d}$ and such that the balls $B_{z}$, $z\in\mathcal{B}_{\ell}$, are disjoint for any $\ell\in\intn{1}{Q(d)}$. Thus
$$
\frac{|m|}{2} \leq m( E^{[0,1)^d} ) \leq \sum_{i=1}^{Q(d)} m\left( \bigsqcup_{z\in\mathcal{B}_{\ell}} B_{z} \right)
$$
so that $m(\bigsqcup_{z\in\mathcal{B}_{\ell}} B_{z})\geq |m|/(2Q(d))$ for some $\ell\in\intn{1}{Q(d)}$. Hence, for some finite set $\mathcal{Z}^{[0,1)^d}\subseteq\mathcal{B}_{\ell}$,
\begin{equation}\label{minmessqcupbz}
m\left( \bigsqcup_{z\in\mathcal{Z}^{[0,1)^d}} B_{z} \right) \geq \frac{|m|}{4Q(d)}.
\end{equation}
For every $z\in\mathcal{Z}^{[0,1)^d}$, let $\lambda_{z}$ denote a $c$-adic cube of smallest generation whose closure is included in $F_{z}$. One easily checks that $\diam{\adherence{\lambda_{z}}}\leq\diam{F_{z}}\leq 12c^2\diam{\adherence{\lambda_{z}}}$. Therefore, $\diam{\adherence{\lambda_{z}}}/\kappa\leq \diam{B_{z}}^t\leq \kappa\diam{\adherence{\lambda_{z}}}$. Moreover, for all distinct $z,z'\in\mathcal{Z}^{[0,1)^d}$, we have
\begin{equation}\label{mindistlambdazz}
d(\adherence{\lambda_{z}},\adherence{\lambda_{z'}}) \geq \frac{1}{4}\max\left( \diam{B_{z}} , \diam{B_{z'}} \right).
\end{equation}
This follows from the fact that the balls $B_{z}$ and $B_{z'}$ are distinct and of radius at most $c^{-j}$ with $j\geq (\log_{c}2)/(t-1)$.

Let $G_{1}=\{\lambda_{z}, z\in\mathcal{Z}^{[0,1)^d}\}$. Every cube $\lambda=\lambda_{z}\in G_{1}$ is associated with the ball $\tilde\lambda=B_{z}$. According to what precedes, we have
\begin{equation}\label{inegcubescantor}
\frac{1}{\kappa}\diam{\adherence{\lambda}}\leq \diam{\tilde\lambda}^t\leq \kappa\diam{\adherence{\lambda}} \qquad\text{and}\qquad d(\adherence{\lambda},\adherence{\lambda'}) \geq \frac{1}{4}\max( \diam{\tilde\lambda} , \diam{\tilde\lambda'} )
\end{equation}
for all distinct $\lambda,\lambda'\in G_{1}$. The closed cubes $\adherence{\lambda}$, for $\lambda\in G_{1}$, form the first level of the generalized Cantor set $K$. Furthermore, let
$$
\forall \lambda\in G_{1} \qquad \pi(\adherence{\lambda}) = \frac{m(\tilde\lambda)}{\sum\limits_{\lambda'\in G_{1}} m(\tilde\lambda')}.
$$
Let $\lambda\in G_{1}$. Note that $\tilde\lambda\supseteq\adherence{\lambda}$ and $m(\tilde\lambda)\leq Mh_{\beta,\ph}(\diam{\tilde\lambda})$. So $m(\tilde\lambda)\leq M\kappa^{\beta/t}h_{\beta/t,\ph}(\diam{\adherence{\lambda}})$ because of~(\ref{inegcubescantor}) and the fact that $r\mapsto r^{-\ph(r)}$ is nonincreasing on $(0,c^{-j}]$. Together with~(\ref{minmessqcupbz}) and~(\ref{majcrocetap1}), this yields
$$
\forall \lambda\in G_{1} \qquad \pi(\adherence{\lambda}) \leq M\kappa^{\beta/t} \frac{4Q(d)}{|m|} h_{\beta/t,\ph}(\diam{\adherence{\lambda}}) \leq h_{\beta/t,\ph+\phi/(1+d)}(\diam{\adherence{\lambda}}).
$$

\subsubsection*{Step 2}

As $\phi/(1+d)\in\Phi$, there is an integer $j\geq\gene{\lambda}_{c}+j_{0}+5$ such that
\begin{equation}\label{majcrocetap2}
\forall r\in (0,c^{-j}] \quad r^{-\phi(r)/(1+d)} \geq \max_{\lambda\in G_{1}} \left( M\kappa^{\beta/t} \frac{4Q(d)}{|m|}\diam{\adherence{\lambda}}^{-\beta} h_{\beta/t,\ph+\phi/(1+d)}(\diam{\adherence{\lambda}})\right).
\end{equation}
Let $\lambda\in G_{1}$. Since $f_{n_{2}}\in D_{c}$ and $m\restr{\lambda}$ and $m^\lambda$ are absolutely continuous with respect to one another, Lemma~\ref{fXSfSigmamespleine} ensures that $E^{\lambda}=f_{n_{2}}(X)\cap S^{\lambda}_{j_{0}}\cap f_{n_{2}}(\Sigma)$ has full $m^\lambda$-measure in $S^{\lambda}_{j_{0}}$. Hence $m^\lambda(E^{\lambda})\geq |m^\lambda|/2$. Adapting the proof of Lemma~\ref{lemfXSfSigma}, it is straightforward to show that for every $z\in E^{\lambda}$, the following properties hold:\begin{itemize}
\item $\exists p_{z}\in\Z^d \quad \|{f_{n_{2}}}^{-1}(z)-p_{z}-x_{i_{z}}\|\leq r_{i_{z}}/2$,
\item $B_{z}=\bar B(z,2|f_{n_{2}}|r_{i_{z}})\subseteq\interieur{\lambda} \quad\text{and}\quad m^\lambda(B_{z})\leq M h_{\beta,\ph}( \diam{B_{z}}/\diam{\adherence{\lambda}} )$
\end{itemize}
for some $i_{z}\in I^{\mu,\alpha,\weak}_{M,\psi}$ with $2r_{i_{z}}\leq c^{-j}$. Let $F_{z}=B(f_{n_{2}}(p_{z}+x_{i_{z}}),(|f_{n_{2}}|r_{i_{z}})^t)$ and note that $F_{z}\subseteq B_{z}\cap f_{n_{2}}(B(p_{z}+x_{i_{z}},{r_{i_{z}}}^t))$. Applying Besicovitch's covering theorem, we obtain a finite set $\mathcal{Z}^{\lambda}\subseteq E^{\lambda}$ such that
\begin{equation}\label{minmessqcupbz2}
m^\lambda\left( \bigsqcup_{z\in\mathcal{Z}^{\lambda}} B_{z} \right) \geq \frac{|m^\lambda|}{4Q(d)}.
\end{equation}
For any $z\in\mathcal{Z}^{\lambda}$, let $\lambda_{z}$ denote a $c$-adic cube of smallest generation whose closure is included in $F_{z}$. We have $\diam{\adherence{\lambda_{z}}}/\kappa\leq \diam{B_{z}}^t\leq \kappa\diam{\adherence{\lambda_{z}}}$ and~(\ref{mindistlambdazz}) holds for any $z'\in\mathcal{Z}^{\lambda}\backslash\{z\}$.

Let $G^\lambda_{2}=\{\lambda_{z}, z\in\mathcal{Z}^{\lambda}\}$. Each cube $\nu=\lambda_{z}\in G^\lambda_{2}$ is associated with the ball $\tilde\nu=B_{z}$. We have $\adherence{\nu}\subseteq\tilde\nu\subseteq\interieur{\lambda}$ along with
$$
\frac{1}{\kappa}\diam{\adherence{\nu}}\leq \diam{\tilde\nu}^t\leq \kappa\diam{\adherence{\nu}} \qquad\text{and}\qquad d(\adherence{\nu},\adherence{\nu'}) \geq \frac{1}{4}\max( \diam{\tilde\nu} , \diam{\tilde\nu'} )
$$
for all distinct $\nu,\nu'\in G^\lambda_{2}$. Let $G_{2}=\bigcup_{\lambda\in G_{1}}G^\lambda_{2}$. The closed cubes $\adherence{\nu}$, for $\nu\in G_{2}$, form the second level of the generalized Cantor set $K$. Furthermore, let
$$
\forall \lambda\in G_{1} \quad \forall \nu\in G^\lambda_{2} \qquad \pi(\adherence{\nu}) = \frac{m^\lambda(\tilde\nu)}{\sum\limits_{\nu'\in G^\lambda_{2}} m^\lambda(\tilde\nu')} \pi( \adherence{\lambda} ).
$$
Let $\lambda\in G_{1}$ and let $\nu\in G^\lambda_{2}$. We have $m^\lambda(\tilde\nu)\leq M h_{\beta,\ph}(\diam{\tilde\nu}/\diam{\adherence{\lambda}})$. Hence $m^\lambda(\tilde\nu)\leq M\kappa^{\beta/t}\diam{\adherence{\lambda}}^{-\beta}h_{\beta/t,\ph}(\diam{\adherence{\nu}})$ because $\diam{\tilde\nu}/\diam{\adherence{\lambda}}\geq\diam{\tilde\nu}\geq\diam{\adherence{\nu}}$ and $r\mapsto r^{-\ph(r)}$ is nonincreasing on $(0,c^{-j}/\diam{\adherence{\lambda}}]$. Together with~(\ref{minmessqcupbz2}),~(\ref{majcrocetap2}) and the bound on $\pi(\adherence{\lambda})$ obtained at the previous step, this results in
$$
\pi(\adherence{\nu}) \leq M\kappa^{\beta/t} \frac{4Q(d)}{|m|}\diam{\adherence{\lambda}}^{-\beta} h_{\beta/t,\ph+\phi/(1+d)}(\diam{\adherence{\lambda}}) h_{\beta/t,\ph}(\diam{\adherence{\nu}}) \leq h_{\beta/t,\ph+\phi/(1+d)}(\diam{\adherence{\nu}}).
$$

\subsubsection*{Summing-up of the construction}

Iterating this procedure, we construct recursively a sequence $(G_{q})_{q\geq 0}$ of collections of $c$-adic cubes which enjoy following properties.
\renewcommand{\theenumi}{\Alph{enumi}}
\begin{enumerate}
\item\label{etapezeroheter} We have $G_{0}=\{[0,1)^d\}$ and $\pi([0,1]^d)=1$. In addition, $[0,1)^d$ contains a finite number of sets of $G_{1}$.
\item\label{bilancantorheter1} For every $q\in\N$, each cube $\nu\in G_{q}$ contains a finite number of sets of $G_{q+1}$. Moreover, there are a closed ball $\tilde\nu$ and a unique cube $\lambda\in G_{q-1}$ such that $\adherence{\nu} \subseteq \tilde\nu \subseteq \interieur{\lambda}$ and $\gene{\nu}_{c}\geq\gene{\lambda}_{c}+j_{0}+5$. We have $\diam{\adherence{\nu}}/\kappa\leq \diam{\tilde\nu}^t\leq \kappa\diam{\adherence{\nu}}$. The center of $\tilde\nu$ belongs to $S^\lambda_{j_{0}}$. The balls $\tilde\nu$, for $\nu\in G_{q}$, are disjoint and $d(\adherence{\nu},\adherence{\nu'}) \geq \max( \diam{\tilde\nu} , \diam{\tilde\nu'} )/4$ for all distinct $\nu,\nu'\in G_{q}$.
\item\label{bilancantorheter2} For every $q\in\N$, the generation of each cube $\lambda\in G_{q}$ is at least $j_{1}$.
\item\label{bilancantorheter3} For every $q\in\N$ and every cube $\nu\in G_{q}$ that is included in $\lambda\in G_{q-1}$, there are $i\in I^{\mu,\alpha,\weak}_{M,\psi}$ with $\diam{\adherence{\nu}}\leq 2r_{i}<\diam{\adherence{\lambda}}$ and $p\in\Z^d$ such that $\adherence{\nu}\subseteq f_{n_{q}}(B(p+x_{i},{r_{i}}^t))$.
\item\label{bilanmesureheter} For every $q\in\N$ and every cube $\nu\in G_{q}$ that is included in $\lambda\in G_{q-1}$,
$$
\pi(\adherence{\nu}) = \frac{m^\lambda(\tilde\nu)}{\sum\limits_{\nu'\in G_{q} \atop \nu'\subseteq\lambda} m^\lambda(\tilde\nu')} \pi( \adherence{\lambda} ) \leq \frac{4Q(d)}{|m|}m^\lambda(\tilde\nu)\pi( \adherence{\lambda} ) \quad\text{and}\quad \pi(\adherence{\nu}) \leq h_{\beta/t,\ph+\phi/(1+d)}(\diam{\adherence{\nu}}).
$$
\end{enumerate}
\renewcommand{\theenumi}{\arabic{enumi}}
Thus $K = \bigcap_{q=0}^\infty\downarrow \bigcup_{\lambda\in G_{q}} \adherence{\lambda}$ is a generalized Cantor set included in $\bigcap_{n}f_{n}(E_{t}^{\weak})$ (because $n_{q}=n$ for infinitely many integers $q$) and $\pi$ can be extended to a Borel probability measure supported on $K$, thanks to Proposition~1.7 in~\cite{Falconer:2003oj}.

\subsubsection*{Scaling properties of $\pi$}

Let $B$ be an open ball with diameter less than the smallest distance between to distinct cubes of $G_{1}$. In addition, suppose that $\diam{B}$ is small enough to ensure that $h_{\beta/t,2\ph+\phi/(1+d)}\in\jauge_{d}$ increases in $[0,\diam{B})$ and that $r^{-\ph(r)}\geq 1$ for all $r\in (0,\diam{B}]$. We can assume that $B$ intersects $K$ and that $B$ intersects the closure of at least two cubes of $G_{q+1}$ for some $q\geq 0$. Otherwise, $\pi(B)$ would vanish owing to~(\ref{bilanmesureheter}). Let $\lambda\in G_{q}$ ($q\geq 0$) be the cube of largest diameter such that $B$ intersects the closure of at least two cubes of $G_{q+1}$ that are included in $\lambda$. Observe that $B$ does not intersect the closure of any other cube of $G_{q}$, so that $\pi(B)\leq\pi(\adherence{\lambda})$. Moreover, $\lambda$ cannot belong to $G_{0}$.

First, assume that $\diam{B}\geq\diam{\adherence{\lambda}}$. As $h_{\beta/t,\ph+\phi/(1+d)}$ increases in $[0,\diam{B})$, we have
$$
\pi(B) \leq \pi(\adherence{\lambda}) \leq h_{\beta/t,\ph+\phi/(1+d)}(\diam{\adherence{\lambda}}) \leq h_{\beta/t,\ph+\phi/(1+d)}(\diam{B})
$$
on account of~(\ref{bilanmesureheter}). Since $\diam{B}^{-\ph(\diam{B})}\geq 1$, this yields
\begin{equation}\label{majpiB1}
\pi(B) \leq h_{\beta/t,2\ph+\phi/(1+d)}(\diam{B}).
\end{equation}

Next, assume that $\diam{B}\leq c^{-j_{0}-5}\diam{\adherence{\lambda}}$. Let $\nu_{1},\ldots,\nu_{p}$ ($p\geq 2$) denote the cubes of $G_{q+1}$ whose closure intersects $B$. Thanks to~(\ref{bilanmesureheter}), we have
$$
\pi(B) = \sum_{p'=1}^p \pi(B\cap\adherence{\nu_{p'}}) \leq \frac{4Q(d)}{|m|}\pi( \adherence{\lambda} ) \sum_{p'=1}^p m^\lambda(\tilde\nu_{p'}).
$$
Let $j$ be the largest integer such that $\diam{B}\leq c^{-j+1}$. Because of~(\ref{bilancantorheter1}), we have $\diam{\tilde\nu_{p'}}\leq c^{-j+3}$ for all $p'\in\intn{1}{p}$ and $d(z,\adherence{\nu_{p'}}) \leq c^{-j+4}$ for all $p'\in\intn{2}{p}$, where $z$ denotes the center of $\tilde\nu_{1}$. Thus
$$
\bigsqcup_{p'=1}^p \tilde\nu_{p'} \subseteq \bar B\left(z,c^{-j+3}+c^{-j+4}\right) \subseteq 3\lambda^c_{j-6}(z)\cap\lambda.
$$
Recall that $j-6\geq\gene{\lambda}_{c}+j_{0}$ and that $z\in S^\lambda_{j_{0}}$, by virtue of~(\ref{bilancantorheter1}). Hence $m^{\lambda}(\lambda^c_{j-6,k})\leq h_{\beta,\ph}(\diam{\lambda^c_{j-6,k}}/\diam{\adherence{\lambda}})$ for each $k\in\intn{0}{c^{j-6}-1}^d$ with $\lambda^c_{j-6,k}\subseteq 3\lambda^c_{j-6}(z)\cap\lambda$. Along with the fact that $c^{-j}<\diam{B}$, this results in
\begin{equation*}\begin{split}
\sum_{p'=1}^p m^\lambda(\tilde\nu_{p'}) &= m^\lambda\left( \bigsqcup_{p'=1}^p \tilde\nu_{p'} \right) \leq 3^d h_{\beta,\ph}\left( \frac{c^{-j+6}}{\diam{\adherence{\lambda}}} \right) \\
&\leq 3^d c^{6\beta} \left( \frac{\diam{B}}{\diam{\adherence{\lambda}}} \right)^{\beta} \left( \frac{c^{-j+6}}{\diam{\adherence{\lambda}}} \right)^{-\ph\left(\frac{c^{-j+6}}{\diam{\adherence{\lambda}}}\right)}.
\end{split}\end{equation*}
Meanwhile, $\diam{B}\leq c^{-j+6}/\diam{\adherence{\lambda}} \leq c^{-j_{0}+1}$ and $r\mapsto r^{-\ph(r)}$ is nonincreasing on $(0,c^{-j_{0}+1}]$, so that the last factor is at most $\diam{B}^{-\ph(\diam{B})}$. As a consequence,
$$
\pi(B) \leq \frac{4Q(d)}{|m|} 3^d c^{6\beta} \left( \frac{\diam{B}}{\diam{\adherence{\lambda}}} \right)^{\beta} \diam{B}^{-\ph(\diam{B})} h_{\beta/t,\ph+\phi/(1+d)}(\diam{\adherence{\lambda}}).
$$
The function $r\mapsto r^{-\ph(r)-\phi(r)/(1+d)}$ is nonincreasing on $(0,c^{-j_{1}}]$ and, by~(\ref{bilancantorheter2}), we have $\diam{B}\leq\diam{\adherence{\lambda}}\leq c^{-j_{1}}$. Thus, for $C=4Q(d) 3^d c^{6\beta}/|m|$,
\begin{equation}\label{majpiB2}\begin{split}
\pi(B) &\leq\frac{4Q(d)}{|m|} 3^d c^{6\beta} h_{\beta/t,2\ph+\phi/(1+d)}(\diam{B}) \left( \frac{\diam{B}}{\diam{\adherence{\lambda}}} \right)^{\beta(1-\frac{1}{t})} \\
&\leq C h_{\beta/t,2\ph+\phi/(1+d)}(\diam{B}).
\end{split}\end{equation}

Now assume that $c^{-j_{0}-5}\diam{\adherence{\lambda}} < \diam{B} < \diam{\adherence{\lambda}}$. Then $B\cap\adherence{\lambda}$ is covered by at most $c^{d(j_{0}+6)}$ cubes of generation $j_{0}+5+\gene{\lambda}_{c}$. Let $\nu$ be such a cube. Note that~(\ref{majpiB2}) still holds for closed balls with diameter less than $c^{-j_{0}-5}\diam{\adherence{\lambda}}$. In particular, if $\adherence{\nu}$ intersects the closure of at least two cubes of $G_{q+1}$ which are included in $\lambda$, we have $\pi(\adherence{\nu}) \leq C h_{\beta/t,2\ph+\phi/(1+d)}(\diam{\adherence{\nu}})$, thanks to~(\ref{majpiB2}). If $\adherence{\nu}$ intersects the closure of only one cube $\lambda'\in G_{q+1}$ which is included in $\lambda$, let $\chi\in G_{q'}$ ($q'\geq q+1$) be the subcube of $\lambda'$ of largest diameter such that $\adherence{\nu}$ intersects the closure of at least two cubes of $G_{q'+1}$ which are included in $\chi$. Then~(\ref{bilancantorheter1}) ensures that $\diam{B}>\diam{\adherence{\nu}}\geq\diam{\adherence{\chi}}$ and~(\ref{majpiB1}) yields $\pi(\adherence{\nu}) \leq h_{\beta/t,2\ph+\phi/(1+d)}(\diam{\adherence{\nu}})$. If $\adherence{\nu}$ does not intersect the closure of any cube of $G_{q+1}$ which is included in $\lambda$, we have $\pi(\adherence{\nu})=0$. It follows that
$$
\pi(B)\leq\max(1,C) c^{d(j_{0}+6)} h_{\beta/t,2\ph+\phi/(1+d)}(c^{-j_{0}-5}\diam{\adherence{\lambda}}).
$$
Recall that $c^{j_{0}}\leq (c^{-j_{1}})^{-\phi(c^{-j_{1}})/(1+d)}\leq \diam{\adherence{\lambda}}^{-\phi(\diam{\adherence{\lambda}})/(1+d)}\leq \diam{B}^{-\phi(\diam{B})/(1+d)}$, because $r\mapsto r^{-\phi(r)/(1+d)}$ is nonincreasing on $(0,c^{-j_{1}}]$. Moreover, $c^{-j_{0}-5}\diam{\adherence{\lambda}}<\diam{B}<\diam{\adherence{\lambda}}\leq c^{-j_{1}}$, owing to~(\ref{bilancantorheter2}), and $h_{\beta/t,2\ph+\phi/(1+d)}$ increases in the same interval. Hence
$$
\pi(B)\leq\max(1,C) c^{6d} h_{\beta/t,2\ph+\phi}(\diam{B}).
$$

We finally obtain~(\ref{eqpropKpi}), for any open ball $B$ of sufficiently small radius. Lemma~\ref{massdistprinc} leads to~(\ref{hauxicapdcstricpos}) for any $t>1$.

By virtue of~(\ref{hauxicapdcstricpos}), $E_{t}^{\weak}$ belongs to $\grint^{h_{\beta,\ph}}(\R^d)$ if $t=1$ and to $\grint^{h_{\beta/t,2\ph+\phi}}(\R^d)$ for any $\phi\in\Phi$ if $t>1$. Lemma~\ref{etweakgrint} implies that $E_{t}$ belongs to the same classes. Moreover, $\tilde E_{t}\cap U=E_{t}\cap U$ for all open subset $U$ of $(0,1)^d$. Thus $\tilde E_{t}$ belongs to $\grint^{h_{\beta,\ph}}((0,1)^d)$ if $t=1$ and to $\grint^{h_{\beta/t,2\ph+\phi}}((0,1)^d)$ for any $\phi\in\Phi$ if $t>1$. Theorem~\ref{UBIQUITYHETER} is proven.

\subsection{Proofs of Proposition~\ref{ubiquitymultinom} and Proposition~\ref{ubiquitygibbs}}

The proofs of Proposition~\ref{ubiquitymultinom} and Proposition~\ref{ubiquitygibbs} are very similar to that of Theorem~\ref{UBIQUITYHETER}, so we just give the necessary modifications.

\subsubsection{Proof of Proposition~\ref{ubiquitymultinom}}

For some Borel set $\Sigma$ of full $\bar m$-measure in $\R^d$ and for all $x\in\Sigma$, there is an integer $j(x)\in\N$ such that for all integer $j\geq j(x)$,
\begin{equation}\label{develcondcad}
\sup_{z\in [0,1]^d\cap 3\lambda^c_{j}(x\modu\Z^d)} \max_{s\in\intn{1}{d} \atop b\in\intn{0}{c-1}} \left|\sigma_{b,j}(z^s)-\pi^s_{b}\right|\leq \sqrt{\log c}\, \varsigma(c^{-j}).
\end{equation}
Note that the set $F$ defined by~(\ref{defFunioncubouv}) has full $\bar m$-measure in $\R^d$. Thus, replacing $\Sigma$ by $\Sigma\cap F$ if necessary, we may assume that $\Sigma\subseteq F$.

We can replace the set $I^{\mu,\alpha,\weak}_{M,\psi}$ by $I^\devel_{\pi}$ in the statement of Lemma~\ref{lemfXSfSigma}. Indeed, with the notations of the proof of Lemma~\ref{lemfXSfSigma}, the point $\dot x=x-q=x\modu\Z^d$ now belongs to $(0,1)^d$ and the integer $j_{0}$ is now such that~(\ref{develcondcad}) holds for any integer $j\geq j_{0}$. Let $i\in I_{2}$ and $p\in\Z^d$ with $x\in\bar B(p+x_{i},r_{i}/2)$. As $x\in F$, we may assume that $\bar B(p+x_{i},r_{i}/2)\subseteq F$. Thus $q=p$ and $\dot x\in\bar B(x_{i},r_{i}/2)$. Let $j_{1}$ be the largest integer such that $r_{i}\leq c^{-j_{1}}$. We have $x_{i}\in\bar B(\dot x,r_{i}/2)\subseteq 3\lambda^c_{j_{1}}(\dot x)\cap [0,1]^d$ and $j_{1}\geq j_{0}$, so that
$$
\max_{s\in\intn{1}{d} \atop b\in\intn{0}{c-1}} \left|\sigma_{b,j_{1}}(x_{i}^s)-\pi^s_{b}\right|\leq \sqrt{\log c}\, \varsigma(c^{-j_{1}})
$$
owing to~(\ref{develcondcad}). Meanwhile, $j_{1}=\lfloor -\log_{c}r_{i} \rfloor$ and $\varsigma(c^{-j_{1}})\leq 2\varsigma(r_{i})$ if $j_{0}$ is large enough. We end up with $i\in I^\devel_{\pi}$.

We can replace $I^{\mu,\alpha,\weak}_{M,\psi}$ by $I^\devel_{\pi}$ everywhere in the proof of Theorem~\ref{UBIQUITYHETER}. This way, we obtain~(\ref{hauxicapdcstricpos}) with $E^{\devel}_{t}$ instead of $E_{t}^{\weak}$, for any $t\geq 1$. Proposition~\ref{ubiquitymultinom} follows.

\subsubsection{Proof of Proposition~\ref{ubiquitygibbs}}

For some Borel set $\Sigma$ of full $\bar m$-measure in $\R^d$ and for all $x\in\Sigma$, there is an integer $j(x)\in\N$ such that for all integer $j\geq j(x)$,
\begin{equation}\label{supencgibbs}
\sup_{z\in 3\lambda^c_{j}(x)} \left| \frac{1}{j} S^{f}_{j}(z) - P'_{f}(q) \right| \leq \frac{\kappa}{|q|}\varsigma(c^{-j}).
\end{equation}
We can replace the set $I^{\mu,\alpha,\weak}_{M,\psi}$ by $I^\birkhoff_{f,q}$ in the statement of Lemma~\ref{lemfXSfSigma}. Indeed, the integer $j_{0}$ is now such that~(\ref{supencgibbs}) holds for any integer $j\geq j_{0}$. Let $i\in I_{2}$ and $p\in\Z^d$ with $x\in\bar B(p+x_{i},r_{i}/2)$. Let $j_{1}$ be the largest integer such that $r_{i}\leq c^{-j_{1}}$. We have $p+x_{i}\in\bar B(x,r_{i}/2)\subseteq 3\lambda^c_{j_{1}}(x)$ and $j_{1}\geq j_{0}$, so that
$$
\left| \frac{1}{j_{1}} S^{f}_{j_{1}}(x_{i}) - P'_{f}(q) \right| \leq \frac{\kappa}{|q|}\varsigma(c^{-j_{1}})
$$
thanks to~(\ref{supencgibbs}) and the periodicity of $f$. Meanwhile, $j_{1}=\lfloor -\log_{c}r_{i} \rfloor$ and $\varsigma(c^{-j_{1}})\leq 2\varsigma(r_{i})$ if $j_{0}$ is large enough. This yields $i\in I^\birkhoff_{f,q}$.

We can replace the set $I^{\mu,\alpha,\weak}_{M,\psi}$ by $I^\birkhoff_{f,q}$ everywhere in the proof of Theorem~\ref{UBIQUITYHETER}. Thus,~(\ref{hauxicapdcstricpos}) holds with $E^{\birkhoff}_{t}$ instead of $E_{t}^{\weak}$, for all $t\geq 1$. Proposition~\ref{ubiquitygibbs} follows.

\end{document}